\documentclass[11pt]{amsart}

\usepackage{amsmath, amsthm}
\usepackage{eucal}

\usepackage{palatino}
\usepackage{euler}

\usepackage{mathrsfs}

\usepackage{amssymb}
\usepackage{amscd}
\usepackage{latexsym}
\usepackage{epsfig}
\usepackage{graphicx}
\usepackage{amsfonts}
\usepackage{psfrag}
\usepackage{caption}
\usepackage{rotating}
\usepackage{mathtools}

\usepackage{fullpage}

\usepackage{pifont}

\usepackage{epsfig}

\newcommand{\acts}{\mathbin{\raisebox{-.5pt}{\reflectbox{\begin{sideways}
$\circlearrowleft$\end{sideways}}}}}

\usepackage{url}
\usepackage{cite}

\usepackage{dsfont}

\input xy
\xyoption{all}
\UseComputerModernTips

\numberwithin{equation}{section}
\numberwithin{figure}{section}
\numberwithin{table}{section}

\newtheorem{theorem}[equation]{Theorem}

\newtheorem{lemma}[equation]{Lemma}
\newtheorem{proposition}[equation]{Proposition}

\newtheorem{Claim}[equation]{Claim}

\theoremstyle{definition}
\newtheorem{definition}[equation]{Definition}
\newtheorem{remark}[equation]{Remark}
\newtheorem{example}[equation]{Example}
\newtheorem{nonexample}[equation]{Nonexample}

\makeatletter
\let\c@equation\c@figure
\makeatother

\makeatletter
\let\c@table\c@figure
\makeatother

\newcommand{\comeq}[1]{\hbox{{\footnotesize #1}}}

\newcommand{\C}{{\mathbb{C}}}

\newcommand{\Z}{{\mathbb{Z}}}
\newcommand{\Q}{{\mathbb{Q}}}
\newcommand{\R}{{\mathbb{R}}}

\newcommand{\T}{{\mathbb{T}}}
\renewcommand{\P}{{\mathbb{P}}}
\renewcommand{\SS}{{\mathbb{S}}}

\newcommand{\X}{{\mathcal{X}}}
\newcommand{\Y}{{\mathcal{Y}}}

\renewcommand{\to}{\longrightarrow}

\newcommand{\algt}{\mathfrak{t}}

\newcommand{\into}{\hookrightarrow}
\newcommand{\onto}{\twoheadrightarrow}

\begin{document}

\title[The topology of toric origami manifolds]{The topology of toric origami manifolds}

\author{Tara S. Holm}
\thanks{Tara Holm was partially supported by Grant \#208975 from the Simons Foundation and NSF Grant DMS--1206466.}
\thanks{Ana Rita Pires was partially supported by an AMS-Simons Travel Grant.}
\address{Department of Mathematics, Malott Hall, Cornell
  University, Ithaca, New York 14853-4201, USA}
\email{tsh@math.cornell.edu}
\urladdr{\url{http://www.math.cornell.edu/~tsh/}}

\author{Ana Rita Pires}
\address{Department of Mathematics, Malott Hall, Cornell
  University, Ithaca, New York 14853-4201, USA}
\email{apires@math.cornell.edu}
\urladdr{\url{http://www.math.cornell.edu/~apires/}}


\date{\today}


\begin{abstract}
A { folded symplectic form} on a manifold is a closed 2-form with the mildest possible degeneracy 
along a hypersurface.  
A special class of folded symplectic manifolds are the {origami} 
symplectic manifolds, 
studied by Cannas da Silva, Guillemin and Pires, who classified {toric origami manifolds} 
 by combinatorial  {origami templates}.  In this paper, we examine the topology of toric origami
 manifolds that have acyclic origami template and co\"orientable folding hypersurface.  We prove that
 the cohomology is concentrated in even degrees, and that the equivariant cohomology satisfies
 the GKM description. Finally we show that toric origami manifolds with co\"orientable folding hypersurface provide a class of examples of Masuda and Panov's torus manifolds.
\end{abstract}

\maketitle

\setcounter{tocdepth}{1}
\tableofcontents

\section{Introduction}\label{sec:intro}

Toric symplectic manifolds are a useful class of examples for testing general theories and making explicit
computations.  Statements and proofs of important theorems often simplify in the case of toric manifolds.  
Delzant's classification of toric symplectic manifolds in terms of convex polytopes allows the translation of
geometric and topological questions into combinatorial ones.  In this paper, we study toric actions in the
category of folded symplectic manifolds.  Relaxing the requirement that the manifold be symplectic broadens
the class of manifolds with toric actions.  The mildest degeneracy is to allow the $2$-form to be zero along
a hypersurface.  In this instance, there remains enough geometric structure to be able to
classify such toric origami manifolds combinatorially.  

In this paper, we study the topology of  a particular class of toric origami manifolds, those with acyclic template
and co\"orientable fold.  For such manifolds, we prove that the ordinary cohomology is concentrated in even
degrees (Theorem~\ref{thm:even cohomology}).  This allows us to deduce a variety of facts about the
equivariant cohomology of these manifolds, and in particular to describe the equivariant cohomology ring
combinatorially (Theorem~\ref{thm:origami GKM}).
Our class of toric origami manifolds does fit into the framework of torus manifolds (Theorem~\ref{thm:locstd}).  The origami structure
allows us to give explicit inductive proofs.  We plan to use similar geometric techniques to study the non-co\"orientable
and non-acyclic cases.  We hope that this approach will also generalize to a class of torus manifolds that arise from
combinatorial origami templates, in the same way that some torus manifolds arise from combinatorial polytopes.

The remainder of this paper is organized as follows.  In Section~\ref{sec:background}, we review the symplectic and folded 
symplectic geometry underlying our work.  We then provide a framework for computing the ordinary and equivariant 
cohomology of origami manifolds with co\"orientable folding hypersurface and acyclic template in Sections~\ref{se:even} and 
\ref{se:eq-coh}.  
In Section~\ref{sec:std} we describe the relationship of our work with the toric topology literature.

\smallskip 

\noindent {\bf Acknowledgements.}  We are grateful to Jean-Claude Hausmann for his help and patience when we
were sorting out the commutativity of diagram~\eqref{eq:comm diag PD}; and to Nick Sheridan for his suggestions
regarding Definition~\ref{def:template}.  We would also like to thank Ana Cannas da Silva, 
Victor Guillemin,  Allen Hatcher, Yael Karshon, Allen Knutson, Tomoo Matsumura, and Milena Pabiniak
for many helpful conversations.   We are very grateful for the comments from the
anonymous referees, which led to several improvements of this article.

\section{Origami manifolds}\label{sec:background}

\subsection{Symplectic manifolds.}\label{se:symplectic}
We begin with a very quick review of symplectic geometry, following \cite{ca:book}.
Let $M$ be a manifold equipped with a {\bf symplectic form} 
$\omega\in \Omega^2(M)$: that is, $\omega$ is closed ($d\omega = 0$) and non-degenerate.
In particular, the non-degeneracy condition implies that $M$ must be an even-dimensional manifold.  The simplest examples include
\begin{enumerate}
\item $M=\SS^2 = \C \P^1$ with $\omega_p (\X,\Y)=$ signed area of the parallelogram spanned 
by $\X$ and $\Y$;

\item $M$ any compact orientable surface with $\omega$  the area form; and

\item $M= \R^{2d}$ with $\omega = \sum dx_i\wedge dy_i$.  The Darboux Theorem says that every symplectic manifold has local co\"ordinates so that
$\omega$ is of this standard form.
\end{enumerate}

Suppose
that a compact connected abelian Lie group $\T= (\SS^1)^n$ acts on $M$ preserving $\omega$.
The action is {\bf weakly Hamiltonian} if for every vector $\xi\in\algt$ in the Lie algebra 
$\algt$ of $\T$, the vector field
$$
\X_\xi(p) = \frac{d}{dt}\Big[ \exp (t\xi)\cdot p \Big] \bigg|_{t=0}
$$
is a {\bf Hamiltonian vector field}. That is, we require $\omega(\X_\xi, \cdot )$ to be an exact one-form\footnote{\, The one-form $\omega(\X_\xi, \cdot )$ is automatically closed because the action preserves $\omega$.}:
\begin{equation}\label{eq:mmap}
\omega(\X_\xi, \cdot ) = d\phi^\xi.
\end{equation}
Thus each $\phi^\xi$ is a smooth function on $M$ defined by the differential equation \eqref{eq:mmap}, so determined 
up to a constant.  Taking them together, we may define a {\bf moment map}
$$
\begin{array}{rcc}
\Phi: M  & \to & \algt^* \\
 p & \mapsto & \left(\begin{array}{rcl}
                \algt & \longrightarrow & \R \\
                \xi & \mapsto & \phi^\xi(p)
                \end{array}\right).
\end{array}
$$

The action is {\bf Hamiltonian} if the moment map $\Phi$ can be chosen to be a $\T$-invariant map. Atiyah and Guillemin-Sternberg have shown that when $M$ is a compact Hamiltonian $\T$-manifold, the image $\Phi(M)$ is a convex polytope,
and is the convex hull of the images of the fixed points $\Phi(M^{\T})$ \cite{at:convexity, gu-st:convexity}.

For an {\bf effective}\footnote{ \, An action is effective if no non-trivial 
subgroup acts trivially.} Hamiltonian $\T$ action on $M$, 
$
\dim(\T)\leq \frac{1}{2}\dim(M).
$
We say that the action is {\bf toric} if this inequality is in fact an equality.  
A symplectic manifold $M$ with a toric Hamiltonian $\T$ action is called a {\bf symplectic toric manifold}.  
Delzant used the moment polytope to classify symplectic toric manifolds.  

A polytope $\Delta$ in $\R^n$ is {\bf simple} if there are $n$ edges incident to each vertex, and it is {\bf rational} 
if each edge vector  has rational slope: it lies in $\Q^n\subset \R^n$.  
A simple polytope is {\bf smooth at a vertex} if the $n$ primitive vectors parallel to the edges at the vertex span the lattice 
$\Z^n\subseteq\R^n$ over $\Z$. It is {\bf smooth} if it is smooth at each vertex.  
A simple rational smooth convex polytope is called a {\bf Delzant polytope}.
We may now state Delzant's result.

\begin{theorem}[Delzant \cite{de:hamiltoniens}]
There is a one-to-one correspondence
$$
\left\{\begin{array}{c}
\mbox{compact toric}\\
\mbox{symplectic manifolds}\\
\end{array}\right\}
\leftrightsquigarrow
\left\{\begin{array}{c}
\mbox{Delzant polytopes}
\end{array}\right\} ,
$$
up to equivariant symplectomorphism on
the left-hand side and affine equivalence on the right-hand side.
\end{theorem}

\subsection{Origami manifolds.}
We now relax the non-degeneracy condition on $\omega$, following \cite{CGP:origami}.
A {\bf folded symplectic form} on a $2n$-dimensional manifold $M$ is a $2$-form
$\omega\in \Omega^2(M)$ that is closed ($d\omega = 0$),  
whose top power $\omega^n$ intersects the zero section
transversely on a subset $Z$
and whose restriction to points in $Z$ has maximal rank.
The transversality forces $Z$ to be a codimension $1$ embedded submanifold of $M$.  We call $Z$
the {\bf folding hypersurface} or {\bf fold}.

The simplest examples of folded symplectic manifolds include the following.
\begin{enumerate}
\item Euclidean space $M= \R^{2d}$ has folded symplectic form $\omega = x_1 dx_1\wedge dy_1 + \sum_{i=2}^d dx_i\wedge dy_i$.  The Folded 
Darboux Theorem says that 
at points in $Z=\{ x_1=0\}$,  every folded symplectic manifold has local co\"ordinates so that
$\omega$ is of this standard form \cite[IIIA.4.2.2]{ma:formes}.

\item Any even-dimensional sphere $M=\SS^{2n}\subset \C^{n}\oplus \R$ may be equipped with the 
form $\omega_{\C^n}\oplus 0$.  The folding hypersurface is
the equator $Z = \SS^{2n-1}\subset \C^n\oplus \{ 0\}$.

\item Any compact surface $M$ can be equipped with a folded symplectic form with $Z$ a union of circles. See, for instance, Example 3.19 of \cite{CGP:origami}, and use Remark 2.33 of the same paper together with the classification of closed surfaces. This includes non-orientable surfaces. For example, $\R P^2$ can be equipped with a folded symplectic form so that $Z$ is a single circle.
\end{enumerate}

Let $i:Z\into M$ be the inclusion of $Z$ as a submanifold of $M$.
Our assumptions imply that $i^*\omega$ has a $1$-dimensional kernel on $Z$.
This line field is called the {\bf null foliation} on $Z$.
An {\bf origami manifold} is a folded
symplectic manifold $(M, \omega)$ whose null foliation
is fibrating: $Z\stackrel{\pi}{\to} B$ is a fiber bundle with orientable circle fibers
over a compact base $B$.
The form $\omega$ is called an {\bf origami form}
and the bundle $\pi$
is called the {\bf null fibration}. A diffeomorphism between two origami manifolds which intertwines the origami forms is called an {\bf origami-symplectomorphism}.
In the examples above, the first is not origami because the fibers are $\R$ rather than $\SS^1$, but the second and third
are origami. 
In the second example, the null fibration is the Hopf bundle $\SS^{2n-1}\to \C P^{n-1}$, and in the
third example, the base $B$ consists of isolated points.

The definition of a Hamiltonian action only depends on $\omega$ being closed.  Thus, in the folded framework, we may
define moment maps and toric actions exactly as in Section~\ref{se:symplectic}.
For example, the action $\T^2\acts \SS^4\subset \C^2\oplus \R$ given by rotation on the 
$\C^2$ co\"ordinates is Hamiltonian with moment map
$$
\Phi ( z_1, z_1, t) = \left( |z_1|^2,|z_2|^2\right) .
$$
The image of this map is shown in Figure~\ref{fig:S4} below.

\begin{figure}[ht]
\centering
{
\includegraphics[scale=0.4]{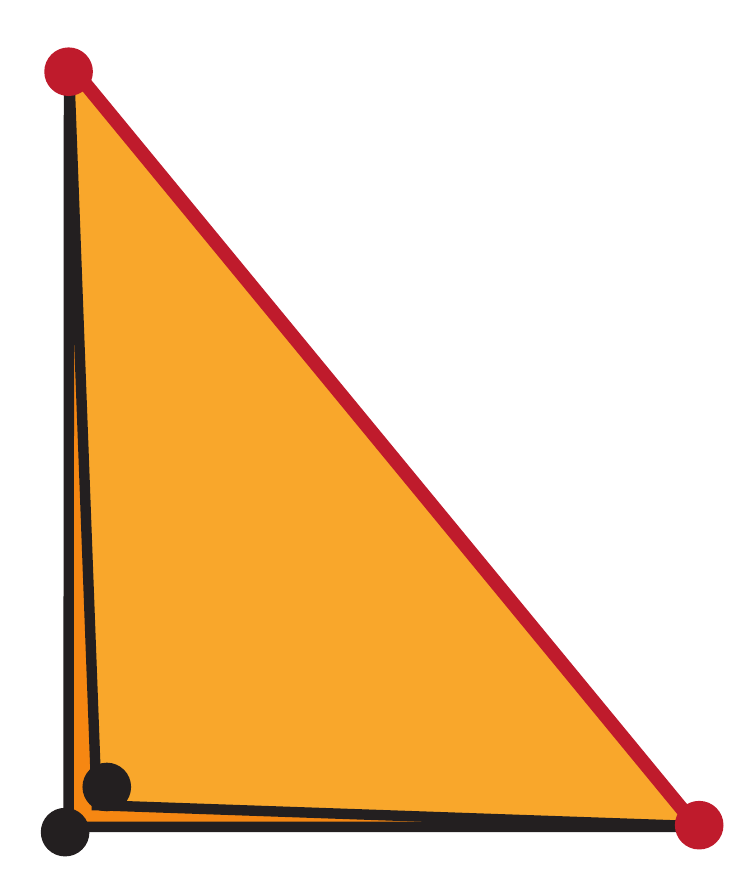} 
}
\caption{The moment map image for the $\T^2$ action on $\SS^4$.  The image consists of two overlapping copies of a triangle, which we have slightly unfolded.  The red hypotenuse is the image of the equator $\SS^3$.
Every other point in the image has two connected components mapping to it, one from the northern hemisphere and the other from the southern.
}
\label{fig:S4}
\end{figure}

An oriented origami manifold $M$ with fold $Z$ may be \textbf{unfolded} into a symplectic manifold as follows. 
Consider the closures of the connected components of $M\setminus Z$, a manifold with boundary
which consists of two copies of $Z$. 
We collapse the fibers of the null fibration by identifying the boundary points that are in the same fiber of the 
null fibration of each individual copy of $Z$. 
The result, $M_0:=(M\setminus Z) \cup B_1 \cup B_2$, is a (disconnected) smooth manifold that can be naturally endowed 
with a symplectic form which on $M_0\setminus (B_1 \cup B_2)$ coincides with the origami form on $M\setminus Z$. 
Because this can be achieved using symplectic cutting techniques, the resulting manifold $M_0$ is called the 
\textbf{symplectic cut space} (and its connected components the \textbf{symplectic cut pieces}), and the process 
is also called \textbf{cutting}. An example of cutting a $2$-torus is shown in Figure~\ref{fig:cut-torus}.
The symplectic cut space of a nonorientable origami manifold is the $\Z_2$-quotient 
of the symplectic cut space of its orientable double cover.

\begin{center}
\begin{figure}[h]
\includegraphics[height=0.65in]{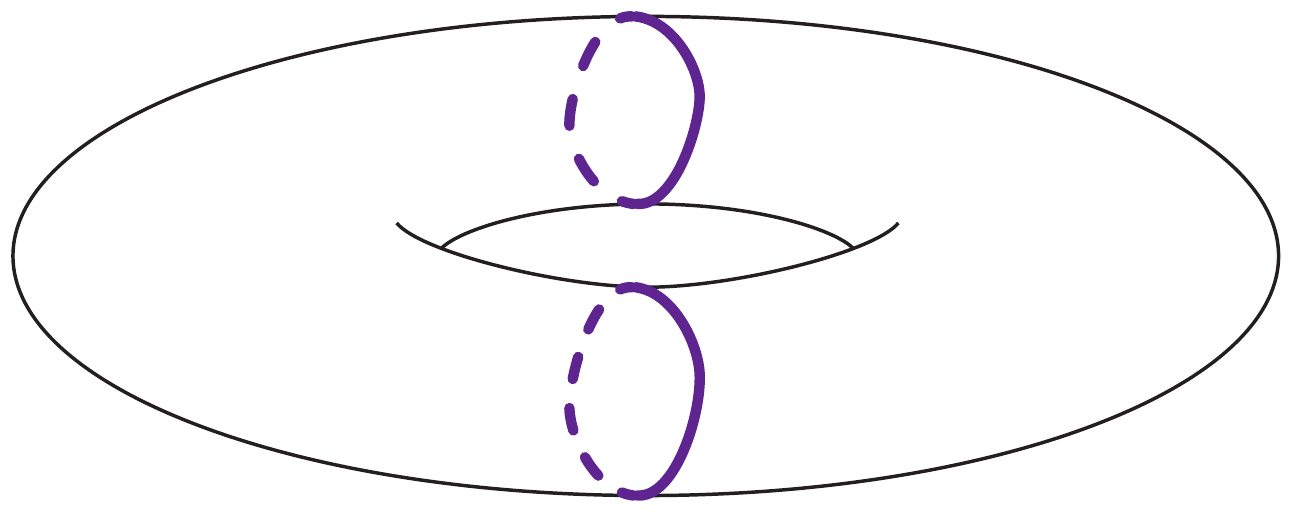} \hskip 0.1in
\includegraphics[height=0.65in]{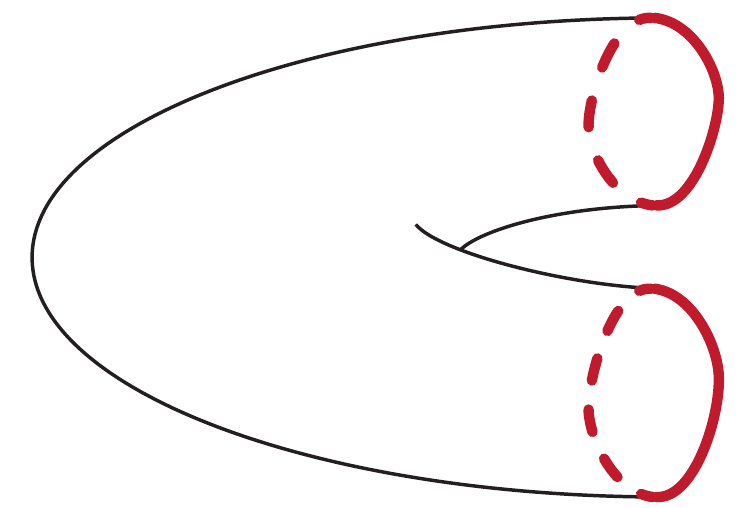} 
\includegraphics[height=0.65in]{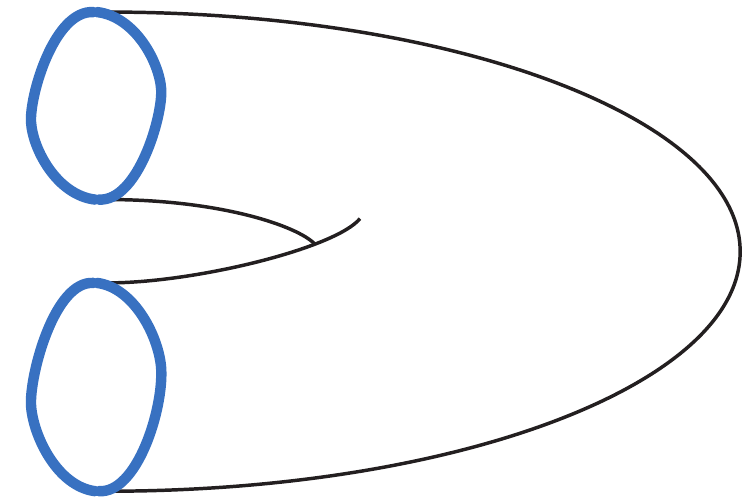} \hskip 0.1in
\includegraphics[height=0.65in]{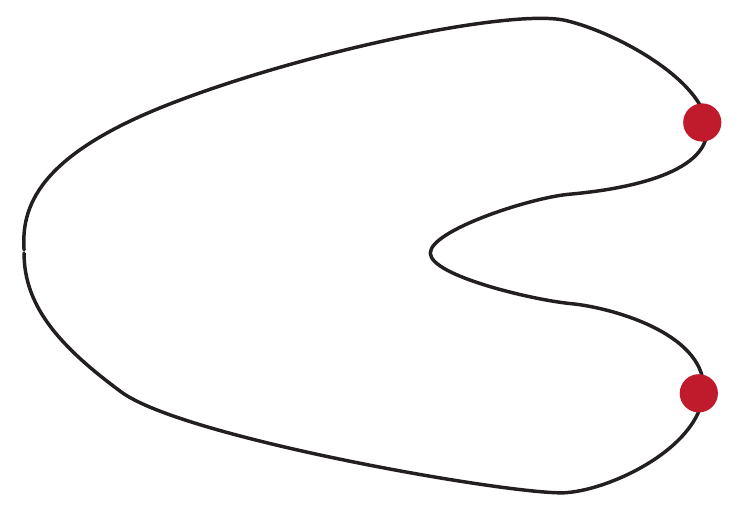} 
\includegraphics[height=0.65in]{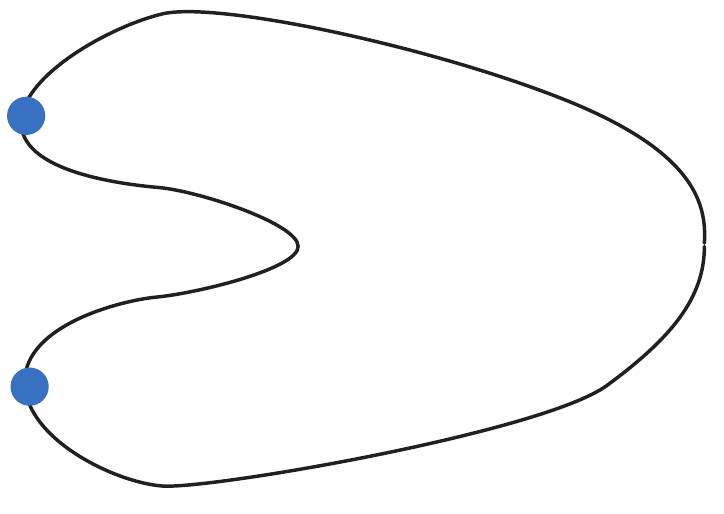} 
\caption{
The torus, with fold $Z=\SS^1\cup \SS^1$ in purple; the middle step before collapsing, 
the two copies of $Z$ are in blue
and purple; and the final cut space $M_0=\SS^2\cup \SS^2$ with $B_1$ in red and $B_2$ in blue.}\label{fig:cut-torus}
\end{figure}
\end{center}

In the example shown in  Figure~\ref{fig:S4}, unfolding the origami $\SS^4$ yields $\C\P^2\sqcup\overline{\C\P^2}$. This is suggested by the image of the moment map: the moment image of each toric $\C\P^2$ (regardless of orientation) is a triangle. The cut space $M_0$ of an oriented origami manifold $(M,\omega)$ inherits a natural orientation. It is the orientation on $M_0$ induced from the orientation on $M$ that matches the symplectic orientation on the symplectic cut pieces corresponding to the subset of $M\setminus Z$ where $\omega^n>0$ and the opposite orientation on those pieces where $\omega^n<0$. In this way, we can associate a $+$ or $-$ sign to each of the symplectic cut pieces of an orientend origami manifold, as well as to the corresponding connected components of $M\setminus Z$.

\begin{remark}\label{rmk:orientable}
In this paper we restrict to origami manifolds whose fold is {\bf co\"orientable}: that is, the fold has an orientable neighborhood. Note that this not imply that the manifold is orientable. Indeed, for an orientable $M$, the condition that $\omega^n$ intersects the zero section transversally implies that the connected components of $M\setminus Z$ which are adjacent in $M$ have opposite signs. Since $M$ is connected, picking a sign for one connected component of $M\setminus Z$ determines the signs for all other components. As a consequence, an origami manifold  $M$ with co\"orientable fold is orientable if and only if it is possible to make such a global choice of signs for the connected components of $M\setminus Z$. 
The moment image of a non-orientable origami manifold 
that nevertheless has co\"orientable fold
is given in Figure~\ref{fig:non-acyclic}.
\end{remark}

\begin{proposition}[\!\! {\cite[Props.\ 2.5 \&  2.7]{CGP:origami}}] \label{prop:model}
Let $M$ be a (possibly disconnected) symplectic manifold with a codimension two symplectic submanifold $B$ and a symplectic involution $\gamma$ of a tubular neighborhood $\mathcal{U}$ of $B$ which preserves $B$\footnote{\, In the nonco\"orientable case, the involution must satisfy  additional conditions, see \cite[Def.\ 2.23]{CGP:origami}. In the co\"orientable case, we have $B=B_1\cup B_2$ and the involution $\gamma$ maps a tubular neighborhood of $B_1$ to one of $B_2$ and vice versa.}.
Then there is an origami manifold $\widetilde{M}$ such that $M$ is the symplectic cut space of $\widetilde{M}$.  Moreover, this manifold is unique up to origami-symplectomorphism. 
\end{proposition}

This newly-created fold $Z\subset\widetilde{M}$ involves the radial projectivized normal bundle of $B\subset M$, so we call the origami manifold $\widetilde{M}$ the \textbf{radial blow-up} of  $M$ through $(\gamma,B)$. 
The cutting operation and the radial
blow-up operation  are in the following sense inverse to each other.

\begin{proposition}[\!\! {\cite[Prop.\ 2.37]{CGP:origami}}]
Let $M$ be an origami manifold with cut space $M_0$.  The radial blow-up
$\widetilde{M_0}$ is origami-symplectomorphic to $M$.
\end{proposition}

There exist Hamiltonian versions of these two operations which may be used to see that the moment map $\Phi$ for an origami manifold $M$ coincides, on each connected component of $M \setminus Z$ with the induced moment map $\Phi_i$  on the corresponding symplectic cut piece $M_i$. As a result, the moment image $\Phi(M)$ is the union of convex polytopes $\Delta_i$.

Furthermore, 
if the circle fibers of the null fibration for a connected component  $\mathcal{Z}$ of the fold $Z$ are orbits for a circle subgroup $\SS^1\subset \T$,
then  
$\Phi(\mathcal{Z})$ is a facet of each of the two polytopes corresponding to neighboring components of $M\setminus Z$.
Let us denote these  two polytopes $\Delta_1$ and $\Delta_2$.  We note that they must \textbf{agree} near $\Phi(\mathcal{Z})$: there is a neighborhood $\mathcal{V}$ of $\Phi(\mathcal{Z})$ in $\R^n$ such that $\Delta_1\cap\mathcal{V}=\Delta_2\cap\mathcal{V}$. The condition that the circle fibers are orbits 
is automatically satisfied when the action is toric, and in that case there is a classification theorem in terms of the moment data.

The moment data of a toric origami manifold can be encoded in the form of an origami template{, originally defined in~\cite[Def.\ 3.12]{CGP:origami}. Definition~\ref{def:template} below is a refinement of that original definition.  The reasons for this refinement are explained in Remark~\ref{rmk:reasons}. 

Following~\cite[p.\ 5]{book:graph}, a \textbf{graph} $G$ consists of a nonempty set $V$ of \textbf{vertices} and a set $E$ of \textbf{edges} together with an incidence relation  that associates an edge with its two \textbf{end vertices}, which need not be distinct. Note that this allows for the existence of (distinguishable) multiple edges with the same two end vertices, and of \textbf{loops} whose two end vertices are equal.
We introduce some additional notation: let $\mathcal{D}_n$ be the set of all Delzant polytopes in $\R^n$ and $\mathcal{E}_n$ the set of all subsets of $\R^n$ which are facets of elements of $\mathcal{D}_n$.  

\begin{definition}\label{def:template}
An $n$-dimensional \textbf{origami template} consists of a graph $G$, called the \textbf{template graph}, and a pair of maps $\Psi_V: V\to\mathcal{D}_n$ and $\Psi_E:E\to\mathcal{E}_n$ such that:
\begin{enumerate}
\item if $e$ is an edge of $G$ with end vertices $u$ and $v$, then $\Psi_E(e)$ is a facet of each of the polytopes $\Psi_V(u)$ and  $\Psi_V(v)$, and these polytopes agree near $\Psi_E(e)$; and

\item if $v$ is an end vertex of each of the two distinct edges $e$ and $f$, then $\Psi_E(e)\cap\Psi_E(f)=\emptyset$.
\end{enumerate}
\end{definition}

The polytopes in the image of the map $\Psi_V$ are the Delzant polytopes of 
the symplectic cut pieces. For each edge $e$, the set $\Psi_E(e)$ is a facet 
of the polytope(s) corresponding to the end vertices of $e$. We refer to such 
a set as a {\bf fold facet}, as it is the image of the connected components of 
the folding hypersurface\footnote{ \, A nonco\"orientable  connected component 
of the folding hypersurface corresponds to a loop edge $e$.}.

In the example of Figure~\ref{fig:S4}, the template graph $G$ has two vertices and one edge joining them. Both vertices are mapped  to the same isosceles right angle triangle under $\Psi_V$, and the edge is mapped to the hypotenuse of that triangle under $\Psi_E$.

\begin{remark}\label{rmk:reasons}
In the original definition of origami template, Definition 3.12 in~\cite{CGP:origami}, a template consisted of a pair $(\mathcal{P},\mathcal{F})$. The set $\mathcal{P}$ was a collection of Delzant polytopes and $\mathcal{F}$ was a collection of pairs or singletons of facets of polytopes in $\mathcal{P}$, satisfying certain conditions. Roughly speaking, $\mathcal{P}$ is the image of $\Psi_V$ and the sets in $\mathcal{F}$ assigned identifications of facets of polytopes in $\mathcal{P}$ in a way similar to that of the map $\Psi_E$. To understand the problem with this old definition we turn again to the example of Figure~\ref{fig:S4}: the collection $\mathcal{P}$ would contain two identical triangles, and $\mathcal{F}$ would contain one pair, consisting of the hypotenuses of each of the triangles. However, $\mathcal{P}$ is a set, and therefore if it consists of two identical elements it actually consists of only one such element. The same issue exists with the pairs in $\mathcal{F}$ and in other examples, with $\mathcal{F}$ itself. Simply replacing the word set by the word multiset to allow for multiple instances of the same element gives rise to a different type of problem.

\vskip 0.1in

\noindent We thank an anonymous referee for bringing this problem to our attention.
\end{remark}
}

\noindent With these combinatorial data in place, we may now state the classification theorem.

\begin{theorem}[\!\! {\cite[Theorem 3.13]{CGP:origami}}]\label{thm:origamiDelzant}
There is a one-to-one correspondence
$$
\left\{\begin{array}{c}
\mbox{compact toric}\\
\mbox{origami manifolds}
\end{array}\right\}
\leftrightsquigarrow
\left\{\begin{array}{c}
\mbox{origami templates}
\end{array}\right\} ,
$$
up to equivariant origami-symplectomorphism on the left-hand side,
and affine equivalence of the image of the template in $\mathbb{R}^n$ on the right-hand side.
\end{theorem}

The orbit space $M/\T $ of a toric origami manifold is closely related to the origami template. When $M$ is a toric symplectic manifold, 
then the orbit space may be identified with the corresponding Delzant polytope; this identification is achieved by the moment map.
For a toric origami manifold, the orbit space is realized as the topological space obtained by gluing the polytopes in $\Psi_V(V)$ 
along the fold facets as specified by the map $\Psi_E$. More precisely, the orbit space is the quotient
\begin{equation}\label{eq:orbitspace}
M/\T = \bigsqcup_{v\in V} (v,\Psi_V(v)) \Big/ \thicksim \ ,
\end{equation}
 where we identify $(u,x)\thicksim(v,y)$ if there exists an edge $e$ with endpoints $u$ and $v$ and the points $x=y\in\Psi_E(e)\subset \R^n$.
 Again, this identification is achieved by the moment map.
In simple low-dimensional examples, we can visualize the orbit space by superimposing the polytopes $\Psi_V(v)$ in 
$\mathbb{R}^n$ and indicating which of their facets to identify; see for instance Figures~\ref{fig:S4}, \ref{fig:non-acyclic}, 
\ref{fig:torictorus} and ~\ref{fig:folded-hirz}.
We will see in Section~\ref{sec:std} that there is a deformation retraction from orbit space $M/\T$ to the template graph.

There is a natural description of the faces of $M/\T$. The  facets of a polytope are well-understood. The set of facets of $M/\T$ is
$$
\bigsqcup_{\mathclap{\substack{v\in V \\ F \text{ facet of } \Psi_V(v)\\ F \text{ not a fold facet}}}} \,(v,F) \Big/ \thicksim \ ,
$$

\vskip 0.05in

\noindent where the equivalence relation is induced by the one in (\ref{eq:orbitspace}). The faces of $M/\T$ are non-empty intersections of facets in $M/\T$, together with $M/\T$ itself.  This notion of face of the orbit space agrees with Masuda and Panov's definition mentioned 
in Section~\ref{sec:std}.

\section{Cohomology concentrated in even degrees}
\label{se:even}

We say that the origami template is {\bf acyclic} if the template graph is acyclic, and therefore a tree. In this case, the {\bf leaves} of the origami template are the polytopes which are images under $\Psi_V$ of the leaves of the template graph.

In light of Remark~\ref{rmk:orientable}, a toric origami manifold with co\"orientable folding hypersurface is orientable exactly when the template graph has no odd cycles. In particular, if $M$ has an acyclic origami template, then $M$ is automatically orientable.
Two non-acyclic origami templates are shown in Figure~\ref{fig:non-acyclic}, one corresponding to an orientable origami manifold and the other to a non-orientable one.

\begin{figure}[h]
\begin{center}
\includegraphics[height=1.6in]{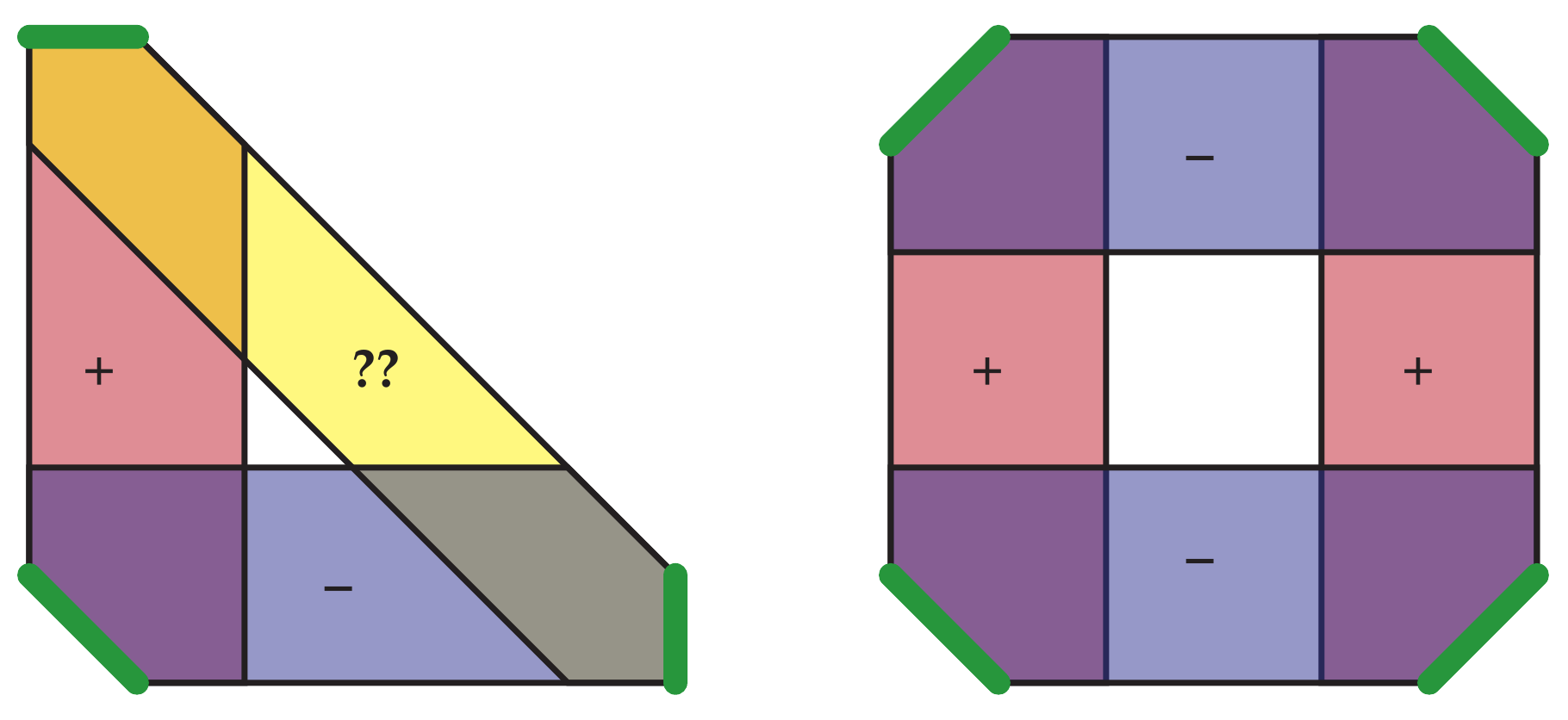}
\caption{The orbit spaces corresponding to two non-acyclic templates of origami manifolds with co\"orientable fold. Such manifolds are orientable exactly when there exists a consistent choice of signs for the polytopes such that the sign changes whenever we traverse a fold facet. The one on the left corresponds to a non-orientable manifold and the one on the right to an orientable manifold.} \label{fig:non-acyclic}
\end{center}
\end{figure}

The proof of the main theorem in this section will involve induction on the number of vertices of the  
template graph.  To prove the inductive hypothesis, we need some auxiliary spaces.  We focus on a connected component
$\mathcal{Z}$ of the fold $Z$ such that $M\smallsetminus\mathcal{Z}$ is the union of one open symplectic manifold
$W_-$ and one open origami manifold $W_+$.  The corresponding closed manifolds with boundary are
$M_-=W_-\cup \mathcal{Z}$ and $M_+=W_+\cup \mathcal{Z}$. Combinatorially,
$M_-$ corresponds to a leaf of the origami template for $M$.

Collapsing the fibers of the null-foliation on $\mathcal{Z}$ results in a toric
symplectic manifold $\mathcal{B}=\mathcal{Z}/\SS^1$ of dimension $\dim(\mathcal{B})=\dim(M)-2$.
Cutting $M$ along $\mathcal{Z}$ yields one toric symplectic manifold $C_-$ and one toric origami
manifold $C_+$ with one fewer connected component of the fold $Z\smallsetminus\mathcal{Z}$.
Finally, we use the space $C=C_+\cup_{\mathcal{B}} C_-$, which is not a manifold.
This notation is illustrated in the Figure~\ref{fig:notation} and summarized in Table~\ref{table:notation}.

\begin{center}
\begin{figure}[h]
\includegraphics[width=0.75in]{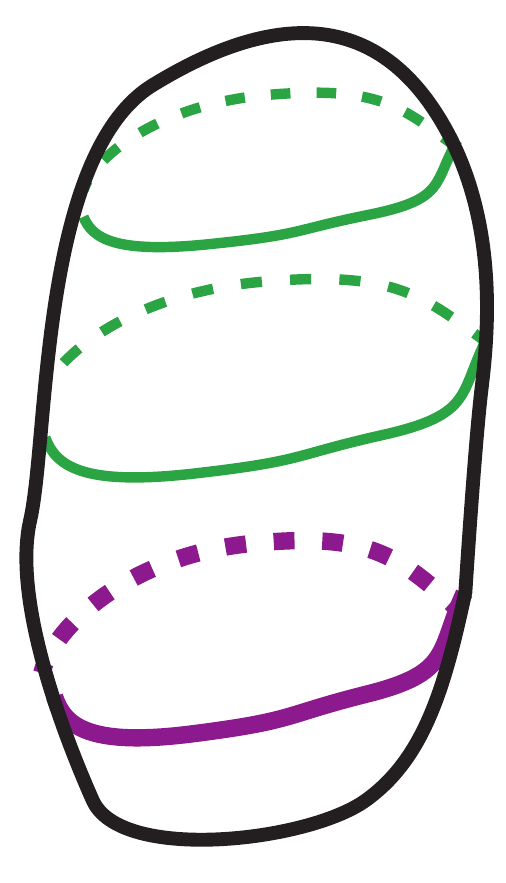} \hskip 0.4in
\includegraphics[width=0.75in]{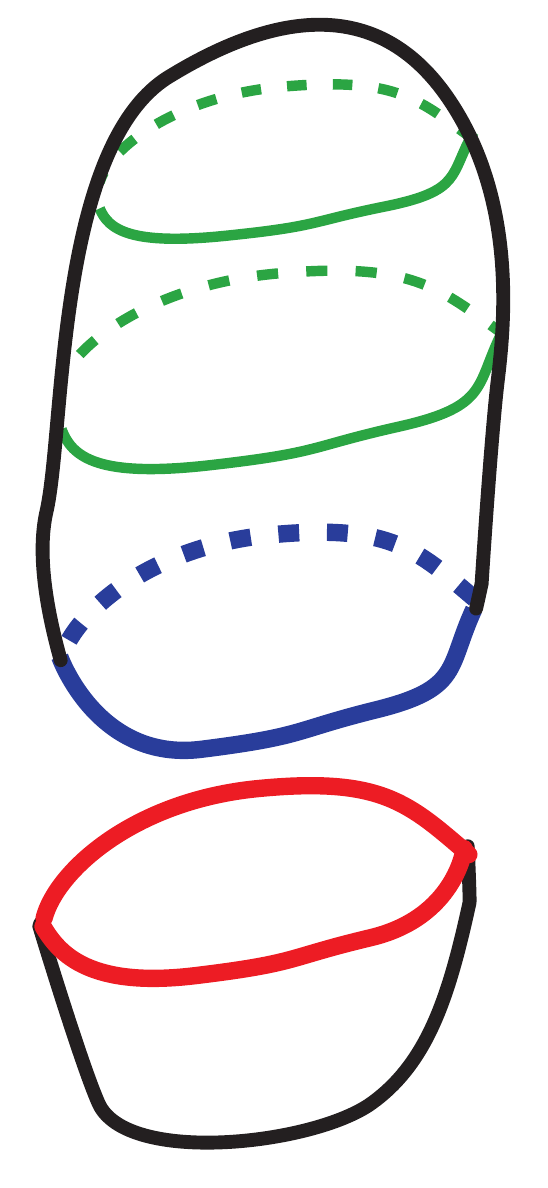} \hskip 0.4in
\includegraphics[width=0.75in]{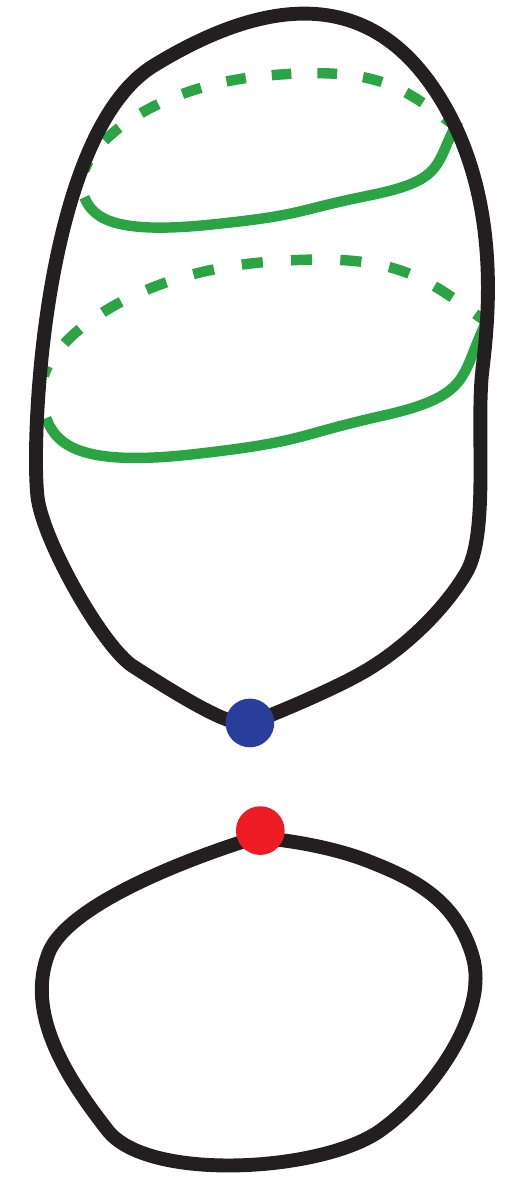} \hskip 0.4in
\includegraphics[width=0.75in]{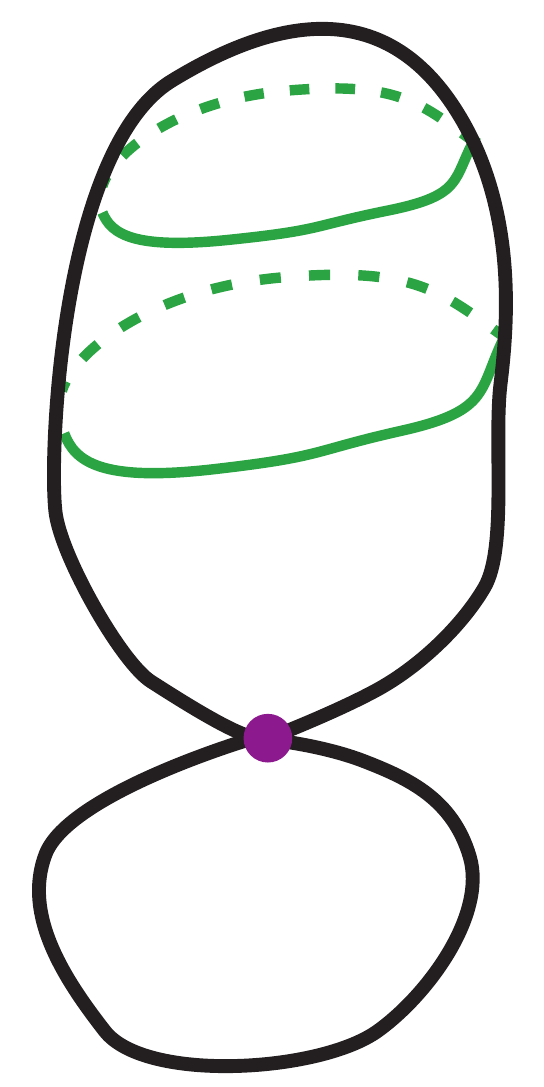} 
\caption{From left to right, the spaces $M$, $M_+\sqcup M_-$, $C_+\sqcup C_-$ and
$C$.}\label{fig:notation}
\end{figure}
\end{center}

\renewcommand{\arraystretch}{1.3}

\begin{center}
\begin{table}[h]\caption{Summary of notation}
\begin{tabular}{c|l}
\hline
Notation & Description  \\ \hline
$M$ & Toric origami manifold, $\T\acts M$ \\ \hline
$\mathcal{Z}\subset Z$ & Connected component $\mathcal{Z}$ of the fold $Z$ \\ \hline
$\mathcal{B}\subset B$ & Toric symplectic manifold $\mathcal{B}=\mathcal{Z}/\SS^1$ and union of such $B=Z/\SS^1$ \\ \hline
$W_+$ & Connected component of $M\setminus \mathcal{Z}$ that is an open origami manifold  \\ \hline
$M_+$ & $W_+\cup \mathcal{Z}$, an origami manifold with boundary  \\ \hline
$C_+$ & $W_+\cup \mathcal{B}$, an origami manifold with one fewer vertex in its template graph \\ \hline
$W_-$ & Connected component of $M\setminus \mathcal{Z}$ that is an open symplectic manifold  \\ \hline
$M_-$ & $W_-\cup \mathcal{Z}$, a symplectic manifold with boundary  \\ \hline
$C_-$ & $W_-\cup \mathcal{B}$, a toric symplectic manifold   \\ \hline
$C$ & $W_+ \cup \mathcal{B}\cup W_-=C_+\cup_{\mathcal{B}} C_-$ (a $\T$-space, but not a manifold) \\ \hline
\end{tabular}\label{table:notation}
\end{table}
\end{center}

\renewcommand{\arraystretch}{1.1}

\begin{lemma}\label{le:magic}
Suppose that $M$ is a compact symplectic toric manifold with moment polytope $\Delta_M$.  
Let $\mathcal{B}$ be a codimension $k$ 
$\T$-invariant symplectic submanifold whose moment map image $\Delta_\mathcal{B}$ is a 
$k$-dimensional face of $\Delta_M$.
Then the inclusion $i:\mathcal{B}\into M$ induces a surjection
$$
i^*:H^*(M;\Z)\onto H^*(\mathcal{B};\Z).
$$
\end{lemma}

\begin{remark}\label{rem:to-masuda-panov}
Though it holds in more generality, we will only use this Lemma when the submanifold $\mathcal{B}$
is of codimension $2$.  
Just as \cite[Lemma 2.3]{masuda-panov} allows Masuda and Panov to make inductive arguments,
our Lemma~\ref{le:magic} will be the crucial ingredient when we build the cohomology of $M$ from
its related toric pieces.
\end{remark}

\begin{proof}
The manifold $\mathcal{B}$ is itself a symplectic toric manifold.  Its cohomology is 
generated in degree $2$, with one 
class for each facet $F$ of $\Delta_\mathcal{B}$.  Such a facet $F$ is the intersection 
of a facet $\widetilde{F}$ of 
$\Delta_M$ with $\Delta_\mathcal{B}$.  
Under the restriction map $i^*$, the generator corresponding to $\widetilde{F}$
maps to the generator corresponding to $F$.  Thus, $i^*$ is surjective.
\end{proof}

\begin{theorem}\label{thm:even cohomology}
Let $\T\acts M$ be a compact toric origami with acyclic origami template and co\"orientable folding hypersurface.
Then the cohomology $H^*(M;\Z)$ is concentrated in even degrees.
\end{theorem}

\begin{proof}
We proceed by induction on the number $n$ of vertices of the template graph, or equivalently, of connected components of $M\setminus Z$.
The base case is when $n=1$ and $M$ is a compact toric symplectic manifold.  In this case, 
the fact that 
$H^*(M)$ is generated in degree 2, and hence concentrated in even degrees is well-known.  For example, see \cite{danilov, jur:toric}.
The case of a connected folding hypersurface is when $n=2$, and concentration in even degrees is proven in 
\cite[Corollary~5.1]{CGP:origami}.

For the inductive step, we assume that every compact toric origami manifold with co\"orientable folding hypersurface and acyclic origami template with at most $(n-1)$ vertices has cohomology 
concentrated in even degrees.  Let $M$ be a compact toric origami manifold with co\"orientable folding hypersurface and acyclic origami template with $n$ vertices.

Choose a leaf of the origami template, and let $\mathcal{Z}$ be the connected component of the folding hypersurface
that corresponds to the facet separating the leaf from the rest of the origami template.  
We use the notation $M_-$, $M_+$, $C_-$, $C_+$, $C$ and $\mathcal{B}$ as 
listed in Table~\ref{table:notation}.
In particular, we note that $C_-$ is actually a compact toric symplectic manifold and $C_+$ is a compact toric origami manifold 
with co\"orientable folding hypersurface and acyclic origami template with $(n-1)$ vertices.

Let $\mathcal{Z}\stackrel{\pi}{\longrightarrow} \mathcal{B}$ be the quotient by the null-fibration.
Then $\pi$ induces maps 
$$
M\stackrel{p}{\longrightarrow} C \mbox{ and } M_-\stackrel{p_-}{\longrightarrow} C_-.
$$

\noindent We begin by studying the cohomology of $C$.

\begin{Claim}\label{cl:H*C}
The cohomology ring $H^*(C;\Z)$ is concentrated in even degrees.
\end{Claim}

\noindent {\bf Proof of Claim~\ref{cl:H*C}.} 
We may choose $\T$-invariant collar neighborhoods of $C_-$ and $C_+$ in $C$ that deformation retract
to $C_-$ and $C_+$ respectively.  This is analogous to choosing a collar neighborhood of $Z$ in $M$,
as described in the remarks just before Proposition~\ref{prop:model} above.

The intersection of these neighborhoods is a collar neighborhood of
$\mathcal{B}$ and deformation retracts onto $\mathcal{B}$.  The Mayer-Vietoris sequence for these collar neighborhoods
induces a long exact sequence, in
cohomology with integer coefficients
\begin{eqnarray}\label{MayerVietoris}
\xymatrix{
\cdots \ar[r] & H^*(C) \ar[r] & H^*(C_+)\oplus H^*(C_-) \ar[r] & H^*(\mathcal{B}) \ar[r] & \cdots
}.
\end{eqnarray}
As $C_-$ is a compact toric symplectic manifold, Lemma~\ref{le:magic} implies that $H^*(C_-)\rightarrow H^*(B)$ is
a surjection.  Thus the long exact \eqref{MayerVietoris} splits into short exact sequences
(again with integer coefficients)
\begin{eqnarray}\label{ShortMayerVietoris}
\xymatrix{
0 \ar[r] & H^*(C) \ar[r] & H^*(C_+)\oplus H^*(C_-) \ar[r] & H^*(\mathcal{B}) \ar[r] & 0
}.
\end{eqnarray}

Note that the cohomology of $C_-$ and $\mathcal{B}$ is concentrated in even degrees because $C_-$ and $\mathcal{B}$
are compact toric symplectic manifolds.  By the induction hypothesis, the cohomology of $C_+$ is
concentrated in even degrees.  We conclude from \eqref{ShortMayerVietoris} in odd degrees that $H^*(C;\Z)$ must be zero in odd degrees.
\hfill \ding{52}

\medskip

\noindent We now look at the relationship between the cohomology of $C_-$ and that of $M_-$.

\begin{Claim}\label{C- surjects}
The quotient map $p_-:M_-\longrightarrow C_-$ induces a surjection in cohomology
$$
p_-^*: H^*(C_-;\Z)\onto H^*(M_-;\Z).
$$
In particular, $H^*(C_-;\Z)$ is concentrated in even degrees, and so $H^*(M_-;\Z)$ is as well.
\end{Claim}

\noindent {\bf Proof of Claim~\ref{C- surjects}.} 
This is an argument based on \cite[Proof of Proposition~1.3]{hk:coh-cuts},
with corrections following \cite{hausmann:personal} and adjustments for integer
coefficients.
Consider long exact sequence in homology with integer coefficients of the pair $(C_-,\mathcal{B})$
\begin{eqnarray}\label{eq:LES homology}
\xymatrix{
\cdots \ar[r] & H_*(\mathcal{B}) \ar[r]^{i_*}  & H_*(C_-) \ar[r]^{j_*} & H_*(C_-,\mathcal{B}) \ar[r] & \cdots,
}
\end{eqnarray}
where $i:\mathcal{B}\into C_-$ is inclusion and $j:({C_-},\emptyset)\to (C_-,\mathcal{B})$ is  inclusion
of the pair. 
We may apply Poincar\'e duality to Lemma~\ref{le:magic} to establish that 
$i_*$ is an injection in homology with integer coefficients.
Thus the long exact sequence \eqref{eq:LES homology} splits into short exact sequences.  
We then have
a commutative diagram, with integer coefficients,
\begin{eqnarray}\label{eq:comm diag PD}
\begin{array}{c}
\xymatrix{
& H^{*-2}(\mathcal{B}) \ar[dd]_{\cong}^{\mbox{\ding{172}}} \ar[r]^{i_!} & H^*(C_-) \ar[dd]_{\cong}^{\mbox{\ding{173}}} \ar[r]^{p_-^*} & H^*(M_-) \ar[d]_{\cong}^{\mbox{\ding{174}}} \\
& & & H_{d-*}(M_-,\mathcal{Z}) \ar[d]_{\cong}^{\mbox{\ding{175}}}\\
0 \ar[r] & H_{d-*}(\mathcal{B}) \ar[r]^{i_*}  & H_{d-*}(C_-) \ar[r]^{j_*} & H_{d-*}(C_-,\mathcal{B}) \ar[r]  & 0.
}\end{array}
\end{eqnarray}
In this diagram, the manifold $C_-$ has dimension $d$, and $\mathcal{B}$ has dimension $d-2$.
The maps \ding{172} and  \ding{173} are Poincar\'e duality for the manifolds $\mathcal{B}$ and $C_-$ respectively, and
 \ding{174} is Poincar\'e duality for the manifold $M_-$ with boundary $\mathcal{Z}$.
Finally, the map \ding{175} is $(p_-)_*$ and is an isomorphism by excision.
The left square commutes because it is the definition of the push-forward map $i_!$.

We now check that the right square commutes.  We use the fact that the Poincar\'e duality isomorphism
is the cap product with the fundamental class.  So we need to show that for any $a\in H^*(C_-)$,
$$
(p_-)_*\big( p^*_-(a)\frown [M_-]\big) = j_*\big( a\frown [C_-]\big).
$$
But now, using the properties of the cap product as developed in \cite[\S 3.3]{hatcher:AT},
we have
$$
\begin{array}{rcll}
(p_-)_*\big( p^*_-(a)\frown [M_-]\big) & =  & a\frown (p_-)_*\big( [M_-]\big) & \comeq{by naturality of the cap product}\\
& = & a\frown j_*\big( [C_-]\big) & \comeq{because  $(p_-)_*\big( [M_-]\big) =j_*\big( [C_-]\big)$}\\
& = & j_*\big(a\frown [C_-]\big) & \comeq{by relative naturality of the cap product, and  $j^*(a)=a$.}
\end{array}
$$

\noindent Thus, the diagram commutes and we may now conclude that $p_-^*$ is a surjection.
\hfill \ding{52}

\medskip

\noindent  Finally, we turn to  the relationship between the cohomology of $C$ and that of $M$.

\begin{Claim}\label{M to C surjects}
The quotient map $p:M\longrightarrow C$ induces a surjection in cohomology
$$
p^*: H^*(C;\Z)\onto H^*(M;\Z).
$$
\end{Claim}

\noindent {\bf Proof of Claim~\ref{M to C surjects}.} 
We have long exact sequences in cohomology with integer coefficients 
for the pairs $(M,M_-)$ and $(C,C_-)$ that fit into a commutative diagram
\begin{eqnarray*} 
\begin{array}{c}
\xymatrix{
\cdots \ar[r] & H^*(C,C_-) \ar[d]_{\cong}^{\mbox{\ding{172}}} \ar[r] & H^*(C) \ar[d]^{\mbox{\ding{173}}}_{p^*} \ar[r] & H^*(C_-) 
\ar@{->>}[d]^{\mbox{\ding{174}}}_{p_-^*} \ar[r] & H^{*+1}(C,C_-) \ar[d]_{\cong}^{\mbox{\ding{175}}} \ar[r]  & \cdots \\
\cdots \ar[r] & H^*(M,M_-) \ar[r] & H^*(M) \ar[r] & H^*(M_-)  \ar[r] & H^{*+1}(M,M_-) \ar[r]  & \cdots .\\
}\end{array}
\end{eqnarray*}
Note that the maps \ding{172} and \ding{175} are isomorphisms by excision, and the map \ding{174} is onto by
Claim~\ref{C- surjects}.
The Four Lemma (the ``onto" half of the Five Lemma) states that if  \ding{172} and \ding{174} are onto and \ding{175} is one-to-one, then \ding{173} 
must be onto.  We have this for each degree, completing the proof.
\hfill \ding{52}

\medskip

Claim~\ref{cl:H*C} guarantees that the cohomology of $C$ is concentrated in even degrees.  Claim~\ref{M to C surjects} tells
us that $H^*(C;\Z)\stackrel{p^*}{\to} H^*(M;\Z)$ is surjective, and so $H^*(M;\Z)$ is necessarily concentrated in even degrees.
\end{proof}

Next we see how the conclusion of Theorem~\ref{thm:even cohomology} can fail
in the non-acyclic case.

\begin{nonexample}\label{eg:toric-torus}
The torus $\T^2$ is a toric origami manifold.  The (toric) circle action is rotation along one of the
coordinate circles.  The folding hypersurface consists of two disjoint circles, as shown in
Figure~\ref{fig:torictorus}. The orbit space consists of two superimposed identical intervals, glued to one another at each end.  The template graph has two vertices (one for each of the intervals) connected to one another by two edges (one for the top fold facet and one for the bottom fold facet), and therefore the template is not acyclic.
\begin{figure}[h]
\begin{center}
\includegraphics[height=1.5in]{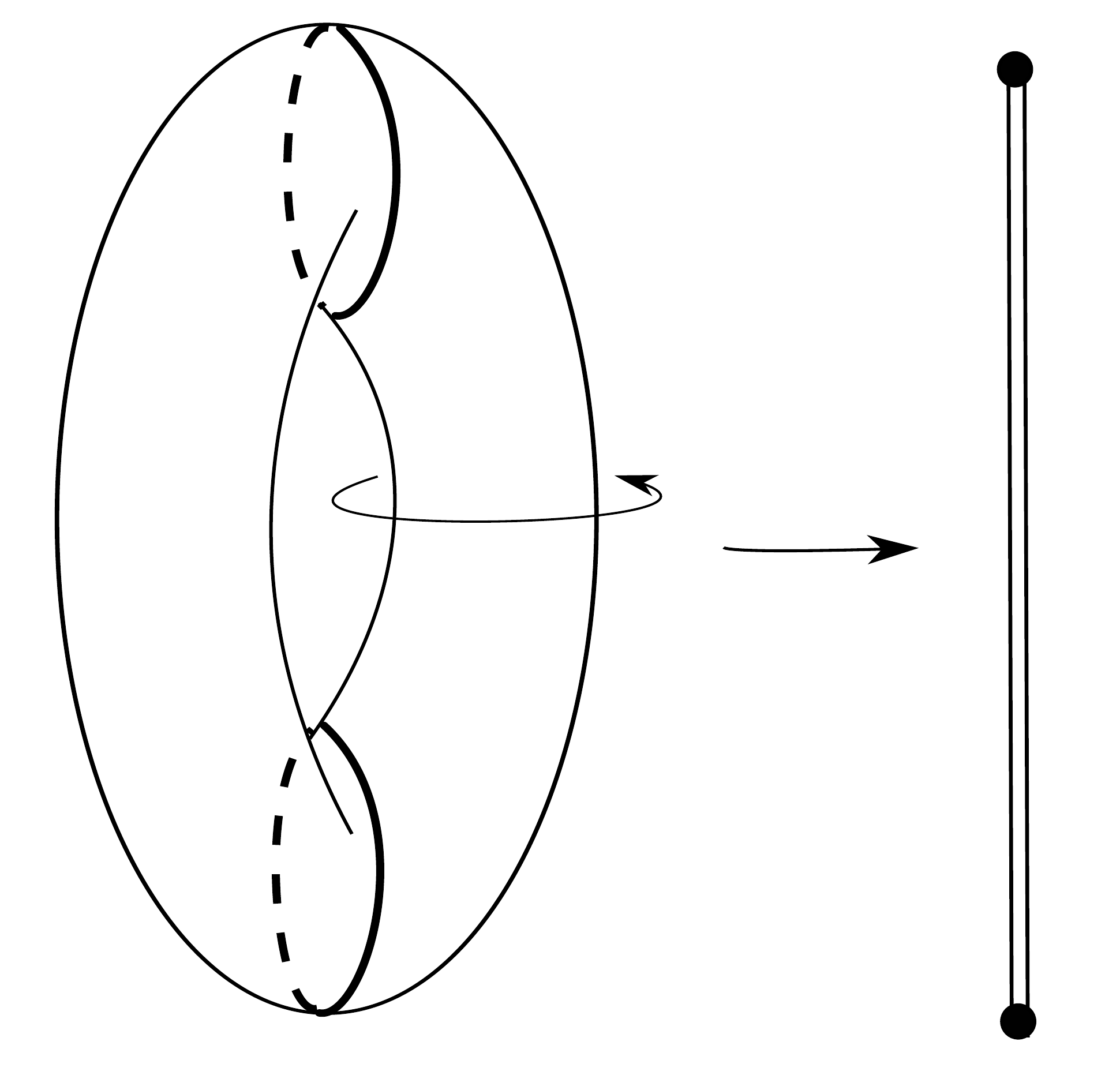}
\caption{The moment map for $\SS^1$ acting on $\T^2$.}\label{fig:torictorus}
\end{center}
\end{figure}

\noindent It is not hard to compute that
$$
H^k(\T^2;\Z) = \left\{\begin{array}{cl}
\Z & k=0,2\\
\Z\oplus \Z & k=1\\
0 & \mbox{else}
\end{array}\right. ,
$$
and so the conclusion of Theorem~\ref{thm:even cohomology} fails. 
\end{nonexample}

\section{Equivariant cohomology}
\label{se:eq-coh}

Equivariant cohomology is a generalized cohomology theory in the equivariant category.  We use the Borel model to compute equivariant cohomology.  For the torus $\T$, we let $E\T$ be a contractible space on which $\T$ acts freely.  Explicitly, for a circle,
we may choose $E\SS^1$ to be the unit sphere $\SS^\infty$ in a Banach space.  This is well-known to be contractible.  Since $\T=\SS^1\times\cdots\times\SS^1$ is a product, we may
 let $E\T$ be a product of infinite-dimensional spheres.

For any $\T$-space $X$, the diagonal action of $\T$ on $X\times E\T$ is free, and 
 $$
 X_{\T} = (X\times E\T)/\T
 $$
is the {\bf Borel mixing space} or {\bf homotopy quotient} of $X$.  We define the (Borel) equivariant cohomology ring to be
 $$
 H_{\T}^*(X;R) := H^*(X_{\T};R),
 $$
 where $H^*(-;R)$ denotes singular cohomology with coefficients in the commutative ring $R$.  Thus, when $X$ is a free ${\T}$-space, we may identify
 $$
 H_{\T}^*(X;R) \cong H^*(X/{\T};R).
$$
At the other extreme, if ${\T}$ acts trivially on $X$, then
 $$
 H_{\T}^*(X;R) \cong H^*(X\times B{\T};R),
$$
where $B{\T} = E{\T}/{\T}$ is the {\bf classifying space} of ${\T}$.  Note that the cohomology of the classifying space, $H^*(B{\T};R) \cong H_{\T}^*(pt;R)$, is the equivariant cohomology ring of a point.

For any $\T$-space $X$, we have the fibration
\begin{equation}\label{eq:fib}
X \hookrightarrow X_{\T} \longrightarrow B\T.
\end{equation}
The projection $X_{\T} \longrightarrow B\T$ induces the map $ H_{\T}^*(pt;R) \to H_{\T}^*(X;R)$, which turns  $H_{\T}^*(X;R)$ into an $H_{\T}^*(pt;R)$-module.  Natural maps in equivariant cohomology preserve this module structure.

A common tool in the computation of equivariant cohomology is the Serre spectral sequence applied to the fibration \eqref{eq:fib}.  
This has $E_2$-page
$$
E_2^{p,q} = H^p(B\T;H^q(X;R)).
$$
This spectral sequence converges to $H_{\T}^*(X;R)$.
When $X$ has cohomology concentrated in even degrees, then this spectral sequence is $0$ in every other row and column,
and automatically collapses.  In particular, the equivariant cohomology is also  concentrated in even degrees.

Goresky, Kottwitz and MacPherson call a $\T$-space $X$ {\bf equivariantly formal}
if the Serre spectral sequence collapses at the $E^2$-page \cite{GKM}.
This spectral sequence does collapse for a compact toric origami
manifold with acyclic origami template and co\"orientable folding hypersurface,
because the cohomology is concentrated in even degrees (Theorem~\ref{thm:even cohomology}).
Historically, the term ``formal'' has been used in rational homotopy theory, and so
equivariantly formal has multiple interpretations.  Scull describes the relationships between
these interpretations \cite{scull}.  To avoid further confusion, we will not use this term in the
remainder of this paper.

Suppose that a torus $\T$ acts on a compact manifold $M$.  Then the 
inclusion of the fixed points $I:M^{\T}\to M$ induces a map in equivariant
cohomology,
\begin{equation}\label{eq:fixed-pts}
I^*:H_{\T}^*(M;R)\to H_{\T}^*(M^{\T};R).
\end{equation}
A classical result of Borel establishes that the kernel and cokernel 
of $I^*$ are torsion submodules \cite{Bo:transf}. 
Our first step is to prove that in our set-up, $I^*$ is injective.
We can deduce this in a variety of ways.  We supply a constructive
proof here that we hope adds geometric intuition in the origami setting.

\begin{theorem}\label{thm:origami inj}
Let $\T\acts M$ be a compact toric origami with acyclic origami template and co\"orientable folding hypersurface.
Then the inclusion $I:M^\T\into M$ induces an injection in equivariant cohomology
$$
I^*:H_\T^*(M;\Z)\to H_\T^*(M^\T;\Z).
$$
\end{theorem}

\begin{proof}
We proceed by induction on the number of vertices in the template graph.

\vskip 0.1in

\noindent {\bf Base Case}:  Suppose the template graph has a single vertex.  Then $M$ is a toric symplectic manifold.  In particular, $M$ is K\"ahler
and has isolated fixed points.
Frankel showed that $H^*(M;\Z)$ is torsion free in this situation 
\cite[Corollary~2]{frankel}.
The Serre spectral sequence then has no torsion at the $E_2$ page, where it collapses, so we may conclude that
$H^*_\T(M;\Z)$ is torsion free.  As the fixed points are isolated,
$H_\T^*(M^\T;\Z)$ is also torsion free, and so Borel's classical result now
implies injectivity.

\vskip 0.1in

\noindent {\bf Inductive Step}:  We now assume that the statement holds for any acyclic toric origami manifold with 
co\"orientable fold with at most $(n-1)$ vertices in its template graph.

As in the previous section, we choose a leaf of the origami template, and let $\mathcal{Z}$ be the
connected component of the folding hypersurface
that corresponds to the facet separating the leaf from the rest of the origami template.  
We continue to use the auxiliary spaces $M_-$, $M_+$, $C_-$, $C_+$, $C$ and $\mathcal{B}$ as 
listed in Table~\ref{table:notation}.

\begin{Claim}\label{GKM for C}
The inclusion $C^\T \to C$ induces an injection
$$
H_\T^*(C;\Z)\to H_\T^*(C^\T;\Z).
$$
\end{Claim}

\vskip 0.1in

\noindent {\bf Proof of Claim~\ref{GKM for C}.} 
We note that $C_-$ is a toric symplectic manifold, and $C_+$ is a toric origami manifold with fewer vertices in its template graph.
Thus, in equivariant cohomology with integer coefficients,
$$
H_\T^*(C_-)\stackrel{I_-^*}{\to} H_\T^*(C^\T_-) \ \mbox{ and } \ H_\T^*(C_+)\stackrel{I_+^*}{\longrightarrow} H_\T^*(C^\T_+)
$$
are both injective.

We now consider the equivariant  Mayer-Vietoris long exact sequence for $\T$-invariant neighborhoods of 
$C=C_+\cup C_-$.  
The spaces $C$, $C_+$, $C_-$ and $\mathcal{B}$ each have ordinary cohomology only in even degrees,
and hence equivariant cohomology only in even degrees.  Thus, the
equivariant Mayer-Vietoris long exact sequence splits into short exact sequences.  
We then have a commutative diagram, with integer coefficients,
\begin{eqnarray*}
\begin{array}{c}
\xymatrix{
0 \ar[r] & H^*_\T(C) \ar[r]^(0.35){\mbox{\ding{173}}} \ar[d]_{\mbox{\ding{172}}} & H_\T^*(C_+)\oplus H_\T^*(C_-) \ar[r]\ar[d]_{\mbox{\ding{174}}} & H_\T^*(\mathcal{B}) \ar[r]\ar[d] & 0\\
0 \ar[r] & H_\T^*(C^\T) \ar[r] \ar[r]_(0.35){\mbox{\ding{175}}} & H_\T^*(C_+^\T)\oplus H_\T^*(C_-^\T) \ar[r] & H_\T^*(\mathcal{B}^\T) \ar[r] & 0
}
\end{array}.
\end{eqnarray*}
The map  \ding{173} is injective because the top row is short exact.    The map   \ding{174}  is $I_-^*\oplus I_+^*$, and is thus injective.  Therefore, 
$\mbox{\ding{174}}\circ\mbox{\ding{173}}$ is injective.  But $\mbox{\ding{174}}\circ\mbox{\ding{173}} = \mbox{\ding{175}}\circ\mbox{\ding{172}}$.  Hence,
\ding{172} must be injective.
\hfill \ding{52}

\begin{Claim}\label{even injectivity}
In even degrees, the map
$$
H^{2*}_\T(C,C_-) \to   H_\T^{2*}(C^\T,C_-^\T)  
$$
is injective.
\end{Claim}

\vskip 0.1in

\noindent {\bf Proof of Claim~\ref{even injectivity}.} 
The pair $(C,C_-)$ is $\T$-invariant, so we consider the long exact sequence of the pair in equivariant cohomology.  
By Claim~\ref{cl:H*C}, the cohomology of $C$ is concentrated in even degrees.  The space $C_-$ is a toric symplectic manifold, so its cohomology is also concentrated in
even degrees.  Thus the long exact sequence splits into a $4$-term short exact  sequence. This induces a commutative diagram
\begin{eqnarray*}
\begin{array}{c}
\xymatrix{
0 \ar[r] & H^{2*}_\T(C,C_-)  \ar[r]^(0.55){\mbox{\ding{173}}} \ar[d]_{\mbox{\ding{172}}} &  H_\T^{2*}(C) \ar[r]\ar[d]_{\mbox{\ding{174}}}  & H_\T^{2*}(C_-) \ar[r]\ar[d]&  H^{2*+1}_\T(C,C_-) \ar[r] \ar[d] & 0 \\
0 \ar[r] & H^{2*}_\T(C^\T,C^\T_-)  \ar[r]^(0.55){\mbox{\ding{175}}}&  H_\T^{2*}(C^\T) \ar[r] & H_\T^{2*}(C^\T_-) \ar[r] &  H^{2*+1}_\T(C^\T,C^\T_-) \ar[r] &  0 \\
}
\end{array}.
\end{eqnarray*}
The map  \ding{173} is injective because the top row is exact.    The map   \ding{174}  is injective by Claim~\ref{GKM for C}.  Therefore, 
$\mbox{\ding{174}}\circ\mbox{\ding{173}}$ is injective.  But $\mbox{\ding{174}}\circ\mbox{\ding{173}} = \mbox{\ding{175}}\circ\mbox{\ding{172}}$.  Hence,
\ding{172} must be injective.
\hfill \ding{52}

\begin{Claim}\label{cl:M_- injectivity}
The inclusion $M_-^\T\into M_-$ induces an injection
$
H_\T^*(M_-)\into H_\T^*(M_-^\T).
$
\end{Claim}

\vskip 0.1in

\noindent {\bf Proof of Claim~\ref{cl:M_- injectivity}.} 
Recall that $C_-$ is a toric symplectic manifold.
Let $f : C_- \to \R$ be the component of its moment map that attains its maximum value on $\mathcal{B}$.  Let $f(\mathcal{B})=b\in \R$.  Let $g: M_-\to \R$ be the composition
$M_- \stackrel{p_-}{\to} C_-\stackrel{f}{\to} \R$.
Choose $\varepsilon>0$ such that there is no critical value in between
$b-\varepsilon$ and $b$, and so that $g^{-1}((b-\varepsilon, b])$ is contained in the intersection of $M_-$ with a Moser neighborhood of $Z$ in $M$.

The fact that $f$ is a Morse-Bott function on $C_-$ with no critical values between $b-\varepsilon$ and $b$ guarantees that $f^{-1}((-\infty, b))$ and  
$f^{-1}((-\infty, b-\frac{\varepsilon}{2}])$ are homotopy equivalent.  In addition, 
the fact that $g^{-1}((b-\varepsilon, b])$ is contained in the intersection of $M_-$ 
with a Moser neighborhood of $\mathcal{Z}$ in $M$ guarantees that $f^{-1}((-\infty, b-\frac{\varepsilon}{2}])$ is homotopy equivalent to $M_-$.

We now appeal to a standard argument from equivariant symplectic geometry to conclude that 
$$
M_-^\T = f^{-1}\left(\left(-\infty, b-\frac{\varepsilon}{2}\right]\right)^\T \into f^{-1}\left(\left(-\infty, b-\frac{\varepsilon}{2}\right]\right) \simeq M_-
$$
induces an injection in equivariant cohomology.  This is an inductive argument on the critical set of $f$, and can be copied verbatim
from the proof of \cite[Theorem 2]{TW:hamTsp}.
\hfill \ding{52}

\vskip 0.1in

We now consider the long exact sequence in equivariant cohomology for the pair $(M,M_-)$. 
We have shown that $M_-$ and $M$ have cohomology and thus equivariant cohomology
concentrated in even degrees.  Thus the long exact sequence splits into a $4$-term short exact sequence.  
This induces a commutative diagram
\begin{eqnarray*}
\begin{array}{c}
\xymatrix{
 0 \ar[r]\ar[d]^{\mbox{\ding{172}}} & H_\T^{2*}(M,M_-) \ar[d]^{\mbox{\ding{173}}} \ar[r] & H_\T^{2*}(M) \ar[d]^{\mbox{\ding{174}}} \ar[r] & H_\T^{2*}(M_-) 
\ar[d]^{\mbox{\ding{175}}} \ar[r] &   H_\T^{2*+1}(M,M_-) \ar[r]\ar[d] & 0\\
0 \ar[r] & H_\T^{2*}(M^\T,M^\T_-) \ar[r] & H_\T^{2*}(M^\T) \ar[r] & H_\T^{2*}(M^\T_-)  \ar[r] &    H_\T^{2*+1}(M,M_-) \ar[r] & 0\\
}\end{array}.
\end{eqnarray*}
We want to show that \ding{174} is injective.  The Four Lemma (the ``injectivity" half of the Five Lemma) states that if  \ding{173} and \ding{175} are injective and \ding{172} is surjective, then \ding{174} 
must be injective.

We first note that $H_\T^{*}(M,M_-)\cong H_\T^{*}(C,C_-)$, and $H_\T^{*}(M^\T,M^\T_-)=H_\T^{*}(C^\T,C^\T_-)$.  Thus, the map \ding{173} is injective 
(in even degrees) by Claim~\ref{even injectivity}.  The map \ding{175} is injective by Claim~\ref{cl:M_- injectivity}.  The map \ding{172} is obviously
surjective.  
Thus, by the Four Lemma, the map \ding{174} must be injective, as desired.
\end{proof}

\begin{remark}
We may also derive Theorem~\ref{thm:origami inj} from work of Franz and Puppe \cite{franz-puppe}.
We describe this approach, and its further applications, in the proof of Theorem~\ref{thm:origami GKM} below.
\end{remark}

We now identify the image of $I^*$.
Goresky, Kottwitz, and MacPherson proved that
the equivariant cohomology of certain spaces may be described
combinatorially as $n$-tuples of polynomials with  divisibility
conditions on pairs of the polynomials \cite[Theorem 1.22]{GKM}.  
The description applies,
for example, to toric varieties  \cite[\S 2.2]{brion}, hypertoric varieties  \cite[Proposition~3.2]{haho}, 
and coadjoint orbits \cite[\S 7.8]{GKM}.  In this section,
we prove that the description also applies to any compact toric origami 
manifold with acyclic origami template and co\"orientable folding hypersurface.
We begin by recalling the assumptions and results from \cite{GKM}.
The two key assumptions are
\begin{enumerate}
\item[(A)] The fixed point set $M^{\T}$ consists of isolated points; and

\item[(B)] The {\bf one-skeleton} $M_{1} = \{ p\in M \ | \ \dim(\T\cdot p)\leq 1\}$
is $2$-dimensional.
\end{enumerate}

The first assumption simplifies what $H_{\T}^*(M^{\T};\Z)$ can be.  
When the fixed point set consists of isolated points, this ring is a direct product of copies of
$$
H_\T^*(pt;\Z)\cong \Z[x_1,\dots,x_n],
$$
one for each fixed point. Thus, every class can be represented as a tuple of polynomials, and the ring structure is the component-wise product of polynomials.

When $M$ is a compact Hamiltonian $\T$-space, the second
assumption ensures that the one-skeleton must consist of $2$-spheres
intersecting one another at the isolated fixed points.
Moreover, the $\T$-action preserves $M_1$, and the action
rotates each $\SS^2$ about an axis.  The image of $M_1$ under the moment
map is an immersed graph $\Phi(M_1) = \Gamma$  called the {\bf moment graph}\footnote{ \, 
The moment graph $\Gamma$ is sometimes called the {\bf GKM graph}.  It is
not the template graph.}
whose vertices correspond to the fixed points $M^{\T}$ and whose edges correspond to the embedded $\SS^2$'s.  Each edge $e$ in $\Gamma$ is labeled by the weight \footnote{ \, This is well-defined up to a sign, which is sufficient for our purposes.} $\alpha_e\in\algt^*$ by which $\T$ acts on $e$.  Indeed, the moment map sends the corresponding $\SS^2$ to a line segment parallel to the weight $\alpha_e$. The embedding of the graph $\Gamma$ encodes in this way the isotropy data, denoted $\alpha$.
In this framework, we have the following description of $H_{\T}^*(M;\Q)$.

\begin{theorem}[Goresky-Kottwitz-MacPherson \cite{GKM}]\label{thm:gkm}
Suppose $M$ is a compact Hamiltonian $\T$-space satisfying conditions {\rm (A)} and {\rm (B)} above.  
Then $I^*$ is injective
$$
I^*:H_{\T}^*(M;\Q)\hookrightarrow H_{\T}^*(M^{\T};\Q)\cong \bigoplus_{p\in M^{\T}} H_{\T}^*(pt;\Q),
$$ 
and its image consists of
\begin{equation}\label{eq:gkm}
 \Big\{ \, (f_p)\in \bigoplus_{p\in M^{\T}} H_{\T}^*(pt;\Q)\ \Big|\ \alpha_e \big| (f_p-f_q) \mbox{ for each edge } e=(p,q) \mbox{ in } \Gamma\, \Big\}.
\end{equation}
We will refer to these divisibility conditions  as the {\bf GKM description}.  
\end{theorem}

\begin{remark}
For a Hamiltonian $\T$-space, assumption (A) guarantees that $I^*$ is injective in equivariant
cohomology with integer coefficients.  We may strengthen assumption (B) to guarantee
that the GKM description holds over $\Z$.  A stronger set of assumptions are described 
in \cite[\S3]{HHH}; they include the existence of a cell decomposition of the manifold.  
In particular, for Hamiltonian $\T$-spaces with isolated points, Morse
theory can be applied to a generic component of the moment map 
to establish that
these stronger assumptions boil down to local topological properties that must be
checked at the fixed points.  These can then be verified for symplectic toric manifolds
and for coadjoint orbits.

As we have seen, the moment map for a toric origami manifold $M$
does not necessarily produce Morse functions on $M$.  We do not know if there
is a cell decomposition of a toric origami manifold that would allow us to apply
techniques from \cite{HHH}.
\end{remark}

A key technical tool in the proof of Theorem~\ref{thm:gkm} is the Chang-Skjelbred Lemma
\cite[Lemma~2.3]{CS:injectivity}.  Let $J: M^{\T}\to M_1$ denote the inclusion of the fixed points into 
the one-skeleton.  The Chang-Skjelbred Lemma states that $I^*(H_{\T}^*(M)) = J^*(H_{\T}^*(M_1))$.
Since the one-skeleton consists of $\SS^2$'s, we must understand $H_\T^*(\SS^2)$.  It is a simple calculation
to check that each $\SS^2$ contributes one of the divisibility conditions in \eqref{eq:gkm}.

Now let $\T\acts M$ be a compact toric origami manifold with acyclic origami template and with co\"orientable folding hypersurface.
The fixed points $M^\T$ correspond to the $0$-dimensional faces of the orbit space $M/\T$.  Just
as for toric symplectic manifolds, these are isolated fixed points.
The one-skeleton corresponds to the (possibly folded) edges ($1$-dimensional faces) of the orbit space. These are the $1$-dimensional faces of the polytopes of the symplectic cut pieces that are not 
entirely contained in a fold. The corresponding subsets of $M$ are symplectic or origami $2$-spheres.
Therefore the one-skeleton is $2$-dimensional.  An example is shown in Figure~\ref{fig:folded-hirz}.

\begin{figure}[h]
\begin{center}
\includegraphics[width=5.75in]{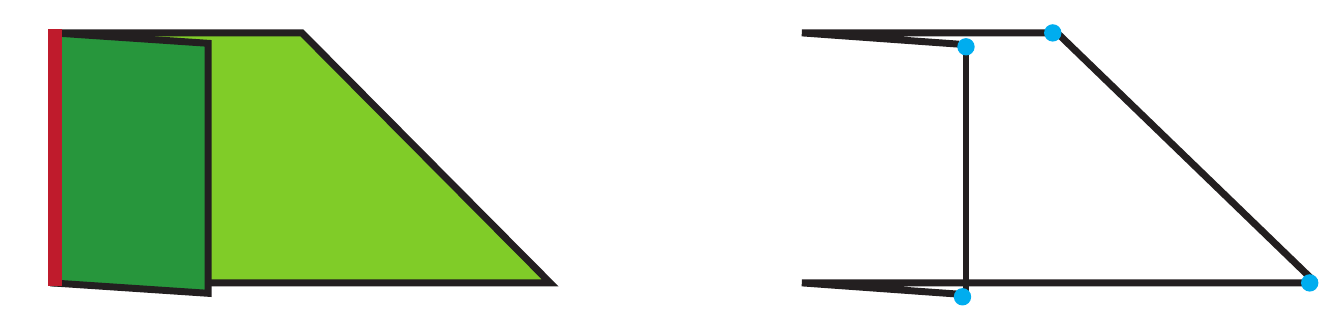}
\caption{
The orbit space and the GKM graph for a toric origami structure on the Hirzebruch surface.
The GKM graph has four vertices and four edges, two of which are folded.
}\label{fig:folded-hirz}
\end{center}
\end{figure}

Thus, assumptions (A) and (B) are satisfied in the case of toric origami manifolds, and indeed the GKM theorem 
generalizes to our set-up.

\begin{theorem}\label{thm:origami GKM}
Let $\T\acts M$ be a compact toric origami with acyclic origami template and co\"orientable folding hypersurface.
Then $I^*$ is injective
$$
I^*:H_{\T}^*(M;\Z)\hookrightarrow H_{\T}^*(M^{\T};\Z)\cong \bigoplus_{p\in M^{\T}} H_{\T}^*(pt;\Z),
$$ 
and the image consists of
\begin{equation}\label{eq:origami gkm}
 \Big\{ (f_p)\in \bigoplus_{p\in M^{\T}} H_{\T}^*(pt;\Z)\ \Big|\ \alpha_e \big| (f_p-f_q) \mbox{ for each edge } e=(p,q) \mbox{ in } \Gamma\Big\},
\end{equation}
where $\alpha_e$ is the weight of the action $\T\acts \SS^2_e$ on the $2$-sphere corresponding to $e$.
\end{theorem}

\begin{proof}
In Theorem~\ref{thm:origami inj}, we have established that $I^*$ is injective (over $\Z$).
This can also be derived from an algebraic result of Franz and Puppe.  In \cite[Theorem~1.1]{franz-puppe}, for a $\T$-space $X$ with connected stabilizers, they show that five conditions are equivalent.
Their condition (ii) is that the Serre spectral sequence collapses at the $E_2$-page.  Their condition
(v) gives a long exact sequence.

A consequence of the origami template classification of toric origami manifolds is that the stabilizer 
of a point is a connected subtorus of $\T$.  Thus, we may appeal to Franz and Puppe's theorem.
Our Theorem~\ref{thm:even cohomology} implies that the Serre spectral sequence collapses
at the $E_2$-page, assertion (ii) in \cite[Theorem~1.1]{franz-puppe}.
This is then equivalent to assertion (v) 
which gives a long exact sequence, the first few terms of which are
$$
      0
      \to H^*_\T(M;\Z)
      \stackrel{\mbox{\ding{172}}}{\to} H^*_\T(M_0;\Z)
      \stackrel{\mbox{\ding{173}}}{\to} H^{*+1}_\T(M_1, M_0;\Z).
$$
The content of our Theorem~\ref{thm:origami inj} is that \ding{172} (which is $I^*$) is injective.  
That the sequence is exact then means that the image of \ding{172} is equal to the kernel
of \ding{173}.  The map \ding{173} is the boundary map in the long exact sequence of the pair
$(M_1,M_0)$.  Thus we have
$$
\cdots \to H_\T^*(M_1^*;\Z) \stackrel{\mbox{\ding{174}}}{\to} H_\T^*(M_0;\Z) \stackrel{\mbox{\ding{173}}}{\to} H_\T^{*+1}(M_1,M_0;\Z)\to\cdots.
$$
The kernel of \ding{173} is then equal to the image of \ding{174}, which is
the image of the equivariant cohomology of the one-skeleton in $H_\T^*(M_0;\Z)$. The fact that
the one-skeleton consists of symplectic and origami $2$-spheres means that each $\SS^2$ contributes 
one of the divisibility conditions in \eqref{eq:origami gkm}.  
\end{proof}

In Section~\ref{se:even}, we proved that $H^*(M;\Z)$ is concentrated in even degrees.  We do not have a Morse
function on $M$ that would allow us to compute the ranks of these cohomology groups.
With our explicit description of $H_\T^*(M;\Z)$, it is possible in examples to determine the ranks and ring structure
of $H^*(M;\Z)$.  This is a consequence of the collapse of the Serre spectral sequence, which implies that
$$
H^*(M;\Z) \cong H_\T^*(M;\Z)\otimes_{H^*_\T(pt;\Z)} \Z.
$$

\begin{example}\label{eg:sphere}
The $2n$-sphere $\SS^{2n}$ may be endowed with toric origami structure whose template graph has two vertices and a single edge between them. Each of the two vertices maps to a the $n$-simplex in $\R^n$ with an orthogonal corner at the origin; that is,
a simplex with vertices the origin and the standard basis vectors $e_i=(0,\dots,1,\dots,0)$ with a single $1$ in the $i^{th}$ co\"ordinate and $0$s elsewhere. The edge maps to the fold facet by which these two polytopes are glued together: the $(n-1)$-simplex with vertices the $e_i$, opposite the
origin.  The orbit space for $\SS^4$ is shown in Figure~\ref{fig:S4}.

Thus the toric action has $2$ fixed points, which we denote $N$ and $S$ (for the north and south poles).  
There are $n$ edges in the GKM graph, each joining $\Phi(N)$ and $\Phi(S)$. 
We can  identify $H_\T^*(pt;\Z) =\Z[x_1,\dots,x_n]$.
From the representation of the orbit space in $\R^n$ we can see that the $\T$-action on the sphere mapping to the
$i^{th}$ co\"ordinate line in $\R^{n}$ has weight $x_i$.
Theorem~\ref{thm:origami GKM} states that
$$
I^*(H_\T^*(\SS^{2n};\Z)) = \Big\{ \, (f_N,f_S)\in \Z[x_1,\dots,x_n]\oplus\Z[x_1,\dots,x_n]\, \Big|
\, x_i\big| (f_N-f_S) \, \mbox{ for } i=1,\dots,n\,\,\Big\}.
$$
From this, we can find a module basis (for $H_\T^*(\SS^{2n};\Z)$ as an $H_\T^*(pt;\Z)$-module)
with two elements
$$
I^*(\mathds{1}) = (1,1) \mbox{ and } I^*(\pi) = (x_1\cdots x_n,0),
$$
where $\mathds{1}\in H_\T^0(\SS^{2n};\Z)$ and $\pi\in H_\T^{2n}(\SS^{2n};\Z)$.

\end{example}

\begin{nonexample}
We revisit Nonexample \ref{eg:toric-torus}, of a toric circle action on a torus. The circle action is free, and so has no fixed points.  Nevertheless, we may 
compute  
$$
H^k_{\SS^1}(\T^2;\Z)= H^k(\T^2/\SS^1;\Z) = H^k(\SS^1) = \left\{\begin{array}{cl}
\Z & k=0,1\\
0 & \mbox{else}
\end{array}\right. .
$$
In particular, the conclusion of Theorem~\ref{thm:origami inj} cannot hold.
\end{nonexample}


\section{Toric origami manifolds are locally standard}\label{sec:std}

Toric topology is the study of topological analogues of toric symplectic manifolds and toric varieties.
The symplectic or algebraic structure is dropped, and the focus is the existence of an effective smooth 
action of a torus half the dimension of the manifold. Examples of such topological analogues,
from most restrictive to most general, are 
toric manifolds \cite{davisjanus} (referred to by some authors as quasitoric manifolds), topological 
toric manifolds \cite{ishidafukukawamasuda} and torus manifolds \cite{masu99}.

We now show that acyclic toric origami manifolds fit into the framework of torus manifolds, and that 
Theorem~\ref{thm:even cohomology} also follows from the work of Masuda
and Panov on the cohomology of torus manifolds \cite{masuda-panov}.  
Their theory is more general and their proofs algebraic.

A {\bf torus manifold} is a $2n$-dimensional closed connected orientable
smooth manifold $M$ with an effective smooth action of an $n$-dimensional
torus $\T^n$ with non-empty fixed set.  
A torus manifold $M$ is said to be {\bf locally standard} if every point in
$M$ has an invariant neighbourhood $U$ weakly equivariantly diffeomorphic to
an open subset $W\subset\C^n$ invariant under the standard $\T^n$-action on
$\C^n$. The adverb `weakly' means that there is an automorphism $\rho\colon \T\to \T$
and a diffeomorphism $f\colon U\to W$ such that $$f(ty)=\rho(t)f(y)$$ for all
$t\in \T$, \ $y\in U$.

Compact symplectic toric manifolds are locally standard \cite[Proof of Lemme 2.4]{de:hamiltoniens}. Next we will prove that toric origami manifolds with co\"orientable folding hypersurface are also locally standard. Toric origami manifolds with non-co\"orientable components of the fold are not locally standard. Indeed, an invariant neighborhood of a point on a non-co\"orientable component of the fold is a bundle of M\"obius bands over the corresponding connected component of $B$ \cite[Rmk.~2.26]{CGP:origami}, which is not equivariantly diffeomorphic to an invariant open subset of $\T^n\acts\mathbb{C}^n$.

\begin{lemma}\label{le:loc-std}
Suppose that $(M,Z,\omega,\Phi,\T)$ is a toric origami manifold 
with co\"orientable folding hypersurface.
Then $M$ is locally standard.
\end{lemma}

\begin{proof}
The argument used in \cite[Proof of Lemme 2.4]{de:hamiltoniens} to prove that compact symplectic toric manifolds are locally standard does not use compacteness of the manifold, and therefore applies directly to the manifold $M\setminus Z$.

Next, we check the `locally standard' condition on a point $p\in Z$ on the fold.  We use a Moser model, as defined in \cite[Def.\ 2.12]{CGP:origami}, for
a neighborhood of $p$.  As remarked in \cite{CGP:origami}, such Moser models exist for orientable origami manifolds.  
What is necessary for the local existence of the Moser model near a single component of the fold is simply the 
co\"orientability of that piece of the fold.  Thus, we may assume that $p\in Z$ has a neighborhood with a Moser model.

Let $\mathcal{Z}_p$ denote the connected component of $Z$ containing $p$.  The local Moser model is an equivariant 
diffeomorphism
$$
\varphi: \mathcal{Z}_p \times (-\varepsilon,\varepsilon) \to \mathcal{U},
$$
where $\varepsilon>0$ and $\mathcal{U}$ is a tubular neighborhood of $\mathcal{Z}_p$, such that $\varphi(x,0)=x$ for all 
$x\in \mathcal{Z}_p$.  The symplectic form can be written in these co\"ordinates, but we do not need that here.

We now consider the null-fibration $\SS^1\hookrightarrow \mathcal{Z}_p\stackrel{\pi}{\longrightarrow} B_p$.  This is a principal 
$\SS^1$-bundle, and the base
space is a compact symplectic toric manifold of dimension $(2n-2)$.  Let $b=\pi(p)$.  Compact toric symplectic manifolds
are locally standard.  Choose a neighborhood $V$ of $b\in B_p$ that is weakly equivariantly diffeomorphic to an open 
subset  $W\subset \C^{n-1}$ that is invariant with respect to the standard $\T^{n-1}$-action on $\C^{n-1}$.  By possibly
passing to a smaller neighborhood of $b$, we may assume that the bundle over $V$ is trivial, 
$V\times \SS^1\stackrel{\pi}{\longrightarrow} V$.  Thus, we have an equivariant neighborhood
$$
V\times \SS^1 \times  (-\varepsilon,\varepsilon)
$$
of $p\in \mathcal{Z}_p$.  Under this identification, the action of $\T^n$ splits into the $\T^{n-1}$ action on $V$, and $\SS^1$ acting
on itself by multiplication on the  $\SS^1$.  We may embed $\SS^1\times (-\varepsilon,\varepsilon)$ as an
open annulus $A\subset \C$ by equivariant diffeomorphism.  Therefore $V\times \SS^1 \times  (-\varepsilon,\varepsilon)$
is weakly equivariantly diffeomorphic to an open 
subset  $W\times A\subset \C^{n-1}\times \C$ that is invariant with respect to the co\"ordinate $\T^{n}$-action on the vector space $\C^{n}$. \footnote{\, An alternative proof for this Lemma was pointed out to us by one of the referees: it uses the fact that any submanifold of $M$ consisting of points with the same isotropy subgroup is tranverse to the folding hypersurface $Z$. This fact relies strongly on the co\"orientability hypothesis.}
\end{proof}

A key player in Masuda and Panov's work on torus manifolds is the orbit space 
$Q =M/\T$, which in the origami framework is closely related to the origami template, as explained at the end of Section~\ref{sec:background}. 
Masuda and Panov define the faces of the orbit space using their notion of characteristic
submanifold.  The orbit space is then called {\bf face-acyclic} if every face $F$ (including $Q$
itself) is {\bf acyclic}: that is, it has $\widetilde{H}^*(F)=0$.

Note that the orbit space $M/\T$ deformation retracts onto the template graph: each polytope $\Psi_V(v)$ deformation retracts onto a point in its center and rays from that point to each of the fold facets of that polytope. This can be done so that when two polytopes are glued along a fold facet, the rays from the center points of the two polytopes join at the fold facet: the two rays now form a line between the center points of the two polytopes. Viewing the center points of the polytopes as vertices of a graph and the lines joining them as edges, we recover the template graph. An example is provided in Figure~\ref{fig:retracts}.

\begin{figure}[ht]
\centering{
\includegraphics[height=1.5in]{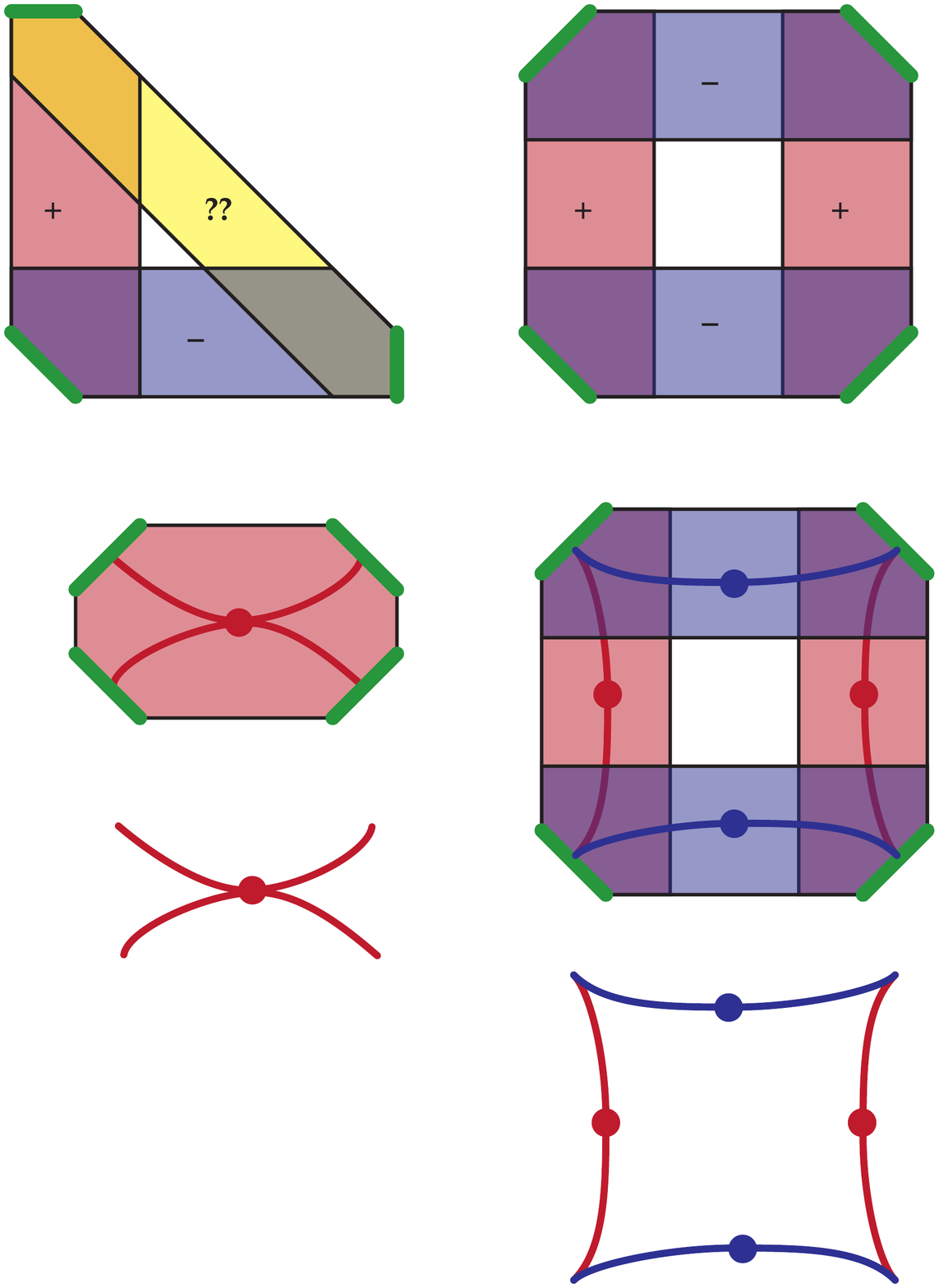} \hskip 0.7in
\includegraphics[height=1.6in]{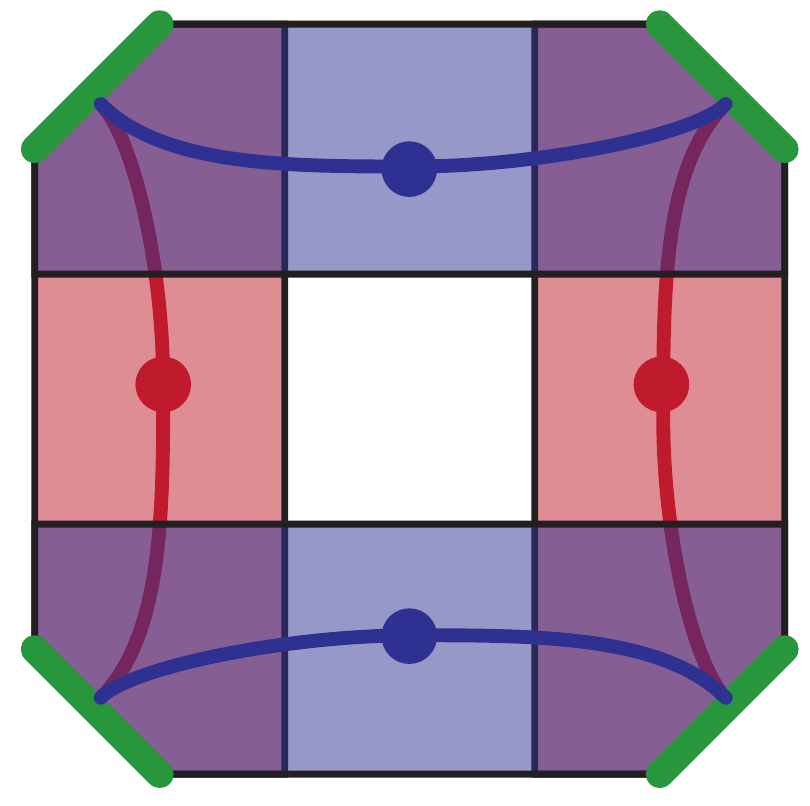}\hskip 0.1in
\includegraphics[height=1.6in]{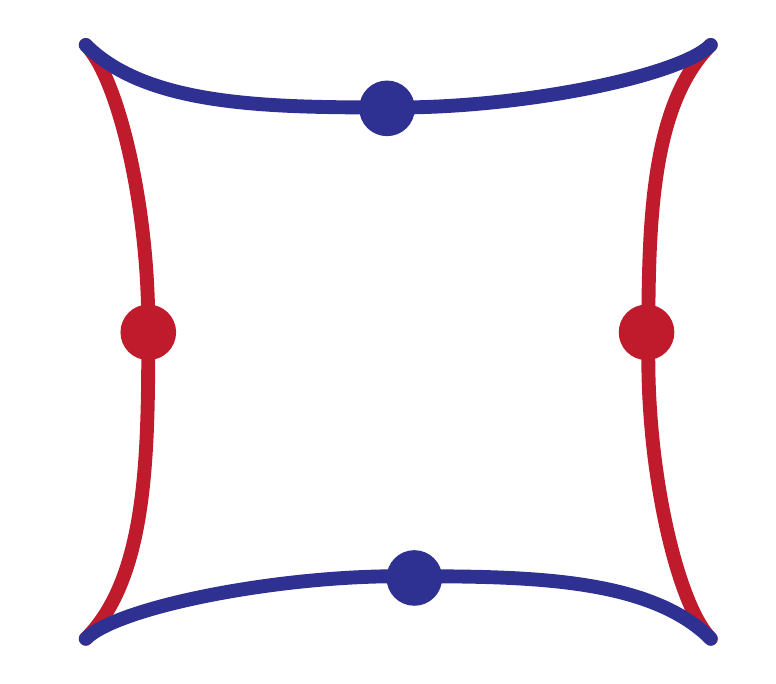}
}
\caption{Each polytope deformation retracts onto a central point and rays towards the fold facets. The orbit space $M/\T$ deformation retracts onto the template graph.}
\label{fig:retracts}
\end{figure}

\begin{theorem}\label{thm:locstd}
Suppose that $(M,Z,\omega,\Phi,\T)$ is a toric origami manifold such that each connected 
component of the folding hypersurface is co\"orientable. Then $M/\T$ is face-acyclic if and
only if 
the origami template is  acyclic.
\end{theorem}

\begin{proof}
The orbit space $M/\T$ deformation retracts onto the template graph, and any face $F$ of $M/\T$ deformation retracts onto a subgraph: the vertices of this subgraph correpond to the polytopes $\Psi_V(v)$ which have non-empty intersection with $F$, its edges are the fold facets $\Psi_E(e)$ which have non-empty intersection with $F$.

Being homotopy equivalent to a (sub)graph, a face of $M/\T$ will be acyclic if and only if that (sub)graph has no cycles, and therefore $M/\T$ is face-acyclic exactly when the template graph has no cycles.
\end{proof}

We now can derive our Theorem~\ref{thm:even cohomology} from Masuda and Panov's work:
they prove that face-acyclic locally standard torus manifolds have no odd-degree cohomology 
\cite[Theorem 9.3]{masuda-panov}. 
While our proofs have very different flavors, it is interesting to note that
a crucial ingredient in their proof is their 
\cite[Lemma 2.3]{masuda-panov}, which is closely related to our Lemma~\ref{le:magic}, as 
described in Remark~\ref{rem:to-masuda-panov}.




\begin{thebibliography}{HHHI}


\bibitem[A]{at:convexity}
M.\ Atiyah, 
``Convexity and commuting Hamiltonians.''
\textit{Bull.\ London Math.\ Soc.}\ \textbf{14} (1982): 1--15.

%

\bibitem[Bo]{Bo:transf}
A.\ Borel, {\em Seminar on transformation groups.} With contributions by G.\ Bredon, E.\ E.\ Floyd, D.\ Montgomery, R.\ Palais.   {\em Annals of Mathematics Studies,} No.\ {\bf 46} Princeton University Press, Princeton, NJ, 1960.  

%

\bibitem[Bri]{brion}
M.\ Brion, 
``Piecewise polynomial functions, convex polytopes and enumerative geometry." {\em Parameter spaces}
 (Warsaw, 1994), 25--44, Banach Center Publ., {\bf 36}, Polish Acad.\ Sci., Warsaw, 1996. 


%
%

\bibitem[C2]{ca:book}
A.\ Cannas da Silva, 
\textit{Lectures on Symplectic Geometry}, Lecture Notes in Mathematics {\bf 1764},
corrected 2nd printing, Springer, 2008.


\bibitem[CGP]{CGP:origami}
A.\ Cannas da Silva, V.\ Guillemin, A.R.\ Pires,
``Symplectic origami."
{\em  IMRN}  {\bf 2011} (2011) no.\ 18: 4252--4293. 

\bibitem[CS]{CS:injectivity} 
T.\ Chang, T.\ Skjelbred,
``Topological Schur lemma and related results.''
 {\em Bull.\ Amer.\ Math.\  Soc.} {\bf 79} (1973): 1036--1038.

\bibitem[Da]{danilov} 
V.\ Danilov,
``The geometry of toric varieties.''
{\em Russian Math.\ Surveys} {\bf 33} (1978) no.\ 2: 97--154.

\bibitem[DJ]{davisjanus}
M.\ Davis, T.\ Januskiewicz,
"Convex polytopes, Coxeter orbifolds and torus actions." 
{\em Duke Math.\ Journal} {\bf 62}  (1991) no.\ 2: 417–-451.

\bibitem[De]{de:hamiltoniens}
T.\ Delzant,
``Hamiltoniens p\'eriodiques et images convexes de l'application moment.''
\textit{Bull.\ Soc.\ Math.\ France} \textbf{116} (1988): 15--339.

\bibitem[Fr]{frankel}
T.\ Frankel,
``Fixed points and torsion on K\"ahler manifolds."
{\em Ann.\ of Math.\ (2)} {\bf 70} (1959): 1--8. 

\bibitem[FP]{franz-puppe}
M.\ Franz and V.\ Puppe, 
``Exact cohomology sequences with integral coefficients for torus actions."
{\em Transform.\ Groups} {\bf  12} (2007), no.\ 1: 65--76. 


\bibitem[GKM]{GKM}
M.\ Goresky, R.\ Kottwitz, R.\ MacPherson,
``Equivariant cohomology, Koszul duality, and the localization theorem.''
{\em Invent.\ Math.} {\bf 131} (1998), no.\ 1: 25--83.

\bibitem[GGL]{book:graph}
R.L.\ Graham, M.\ Gr\"otschel, L.\ Lov\'asz, ed.\ 
\textit{Handbook of Combinatorics}, Volume 1, MIT Press, 1995.


\bibitem[GLS]{gu-le-st:multiplicity}
V.\ Guillemin, E.\ Lerman, S.\ Sternberg,
\textit{Symplectic Fibrations and Multiplicity Diagrams},
Cambridge University Press, Cambridge, 1996.

\bibitem[GS]{gu-st:convexity}
V.\ Guillemin, S.\ Sternberg,
``Convexity properties of the moment mapping.''
\textit{Invent.\ Math.}\ \textbf{67} (1982): 491--513.

\bibitem[HHH]{HHH}
M.\ Harada, A.\ Henriques, T.\ Holm, 
``Computation of generalized equivariant cohomologies of Kac-Moody flag varieties.'' 
 {\em Adv.\ Math.}  {\bf 197}  (2005),  no.\ 1: 198--221.  
 
 

\bibitem[HH]{haho}
M.\ Harada and T.\ Holm,
``The equivariant cohomology of hypertoric varieties and their real loci."
{\em Comm.\ Anal.\ Geom.\ } {\bf 13} (2005), no. 3, 527--559. 

\bibitem[Hat]{hatcher:AT}
A.\ Hatcher,
\emph{Algebraic Topology}, Cambridge University Press (2002). 


\bibitem[Hau]{hausmann:personal}
J.-Cl.\ Hausmann, personal communication.

\bibitem[HK]{hk:coh-cuts}
J.-Cl.\ Hausmann and A.\ Knutson, 
``Cohomology rings of symplectic cuts." 
{\em Differential Geom.\ Appl.}  {\bf 11} (1999), no.\ 2: 197--203.

%

\bibitem[IFM]{ishidafukukawamasuda}
H.\ Ishida, Y.\ Fukukawa and M.\ Masuda,
``Topological toric manifolds.''
\textit{Mosc.\ Math.\ J.} \textbf{13} (2013) no.\ 1: 57--98.

\bibitem[J]{jur:toric}
J.\ Jurkiewicz,
``Chow ring of projective nonsingular torus embedding."
{\em Colloq.\ Math.} {\bf 43} (1980), no.\ 2: 261--270. 

%
%
%

\bibitem[Mas]{masu99}
M.\ Masuda,
``Unitary toric manifolds, multi-fans and equivariant index.''
\textit{Tohoku Math.\ J.} \textbf{51}:2 (1999): 237--265.

\bibitem[MP]{masuda-panov}
M.\ Masuda, T.\ Panov,
``On the cohomology of torus manifolds." 
{\em Osaka J.\ Math.}  {\bf 43} (2006), no.\ 3: 711--746. 

\bibitem[Mar]{ma:formes}
J.\ Martinet,
``Sur les singularit\'es des formes diff\'erentielles.''
\textit{Ann.\ Inst.\ Fourier (Grenoble)} \textbf{20} (1970): 95--178.


\bibitem[S]{scull}
L.\ Scull, ``Equivariant formality for actions of torus groups." 
{\em Canad.\ J.\ Math.\ } {\bf 56} (2004), no.\ 6: 1290--1307.

\bibitem[TW]{TW:hamTsp}
S.\ Tolman, J.\ Weitsman,  
``On the cohomology rings of Hamiltonian $T$-spaces." 
{\em Northern California Symplectic Geometry Seminar}, part
of {\em Amer.\ Math.\ Soc.\ Transl.\ Ser.\ 2}, {\bf 196}, Amer.\ Math.\ Soc., Providence, RI (1999): 251--258.

\end{thebibliography}
\end{document}